\documentclass[12pt,a4paper]{amsart}
 \usepackage[T1]{fontenc}
\usepackage[margin=3cm]{geometry}           		
											
\usepackage{amssymb}
\usepackage{mathtools}
\usepackage{verbatim,enumerate}
\usepackage{xcolor}

\usepackage[looser]{newtxtext}
\usepackage[smallerops]{newtxmath}

\usepackage[hidelinks]{hyperref}
\allowdisplaybreaks

\newtheorem{defn}{Definition}[section]
\newtheorem{thm}[defn]{Theorem}
\newtheorem{lem}[defn]{Lemma}
\newtheorem{exam}[defn]{Example}
\newtheorem{prop}[defn]{Proposition}
\newtheorem{coro}[defn]{Corollary}

\newtheorem{rem}[defn]{Remark}

\newcommand{\bbn}{\mathbb{N}}

\newcommand\dd{\mathop{}\!\mathrm{d}}

\usepackage{etoolbox}
\makeatletter
\patchcmd{\@settitle}
    {\uppercasenonmath\@title}   
    {}
    {}{\typeout{failed}}
\patchcmd{\@setauthors}
    {\MakeUppercase{\authors}}  
    {\authors}
    {}{\typeout{failed}}
\makeatother

\title[Precompactness, zero maximal pattern entropy and bounded mean complexity]{The equivalence of precompactness, zero maximal pattern entropy and bounded mean complexity for finite partitions}
\author[J. Li]{Jian Li}
\address[J. Li]{Institute for  Mathematical Sciences and Artificial Intelligence \& Department of Mathematics,
	Shantou University, Shantou, 515821, Guangdong, China}
\email{lijian09@mail.ustc.edu.cn}
\urladdr{https://orcid.org/0000-0002-8724-3050}

\author[T. Yu]{Tao Yu}

\address[T. Yu]{Department of Mathematics,
	Shantou University, Shantou, 515821, Guangdong, China}
\email{ytnuo@mail.ustc.edu.cn}

\author[X. Zhong]{Xianliang Zhong}
\address[X. Zhong]{Department of Mathematics,
	Shantou University, Shantou, 515821, Guangdong, China}
\email{xlzhong25@163.com}

\date{\today}



\begin{document}

\begin{abstract}
In this paper, we investigate several types of low complexity of finite partitions, including precompactness, zero maximal pattern entropy, bounded mean complexity and mean equicontinuity.
We first show that a collection of finite partitions in a standard probability space is precompact in the Rokhlin metric 
if and only if it has zero maximal pattern entropy 
if and only if the collection of the characteristic functions of atoms in those partitions is precompact in $L^2$ 
if and only if it has bounded mean complexity with respect to the Hamming distance.
Next, we show that for a countably infinite discrete amenable group acting on a standard probability space,
a finite partition has zero maximal pattern entropy 
if and only if each characteristic function 
of an atom in the partition is almost periodic 
if and only if it has bounded mean complexity with respect to some (and hence any) F{\o}lner sequence 
if and only if 
it is mean equicontinuous with respect to some (and hence any) tempered F{\o}lner sequence.
\end{abstract}

\maketitle

\section{Introduction}
In ergodic theory, measure-preserving systems with discrete spectrum represent the simplest class of aperiodic systems.
Characterizing discrete spectrum systems from various complexity perspectives leads to fruitful outcomes.
In \cite{HN42}, Halmos and von Neumann proved that 
an ergodic measure-preserving transformation has discrete spectrum if and only if it is metrically isomorphic to a rotation on a compact abelian group, 
and the spectrum itself is a complete invariant for metric isomorphism within this class.
By the Koopman-von Neumann spectrum theorem, 
a measure-preserving transformation has discrete spectrum if and only if every square integrable function is almost periodic, 
that is, the orbit of every square integrable function under the Koopman operator is precompact in $L^2$, 
see e.g. \cite[Theorem 3.12]{G03}.
In \cite{K67}, Ku\v{s}nirenko demonstrated that a  measure-preserving transformation has discrete spectrum if and only if its measure-theoretical sequence entropy is zero with respect to any infinite sequence of positive integers.  
In~\cite{F97}, Ferenczi studied the measure-theoretic
complexity of finite partitions in ergodic systems via the Hamming distance and showed that an ergodic measure-preserving transformation has discrete spectrum if and only if the complexity function is bounded for any partition. 
In~\cite{Y19}, Yu showed that Ferenczi's result holds without the assumption of ergodicity.
In \cite{HY09} Huang and Ye introduced the notion of maximal pattern entropy, 
and showed that a measure-preserving transformation has discrete spectrum if and only if it has zero maximal pattern entropy, 
since the maximal pattern entropy is the supremum of the measure-theoretical sequence entropy ranging over all infinite sequences of positive integers.
In \cite{VZP13} Vershik et al.\@ studied the complexity of measure-preserving system via admissible semi-metrics, and showed that  a measure-preserving transformation has discrete spectrum if and only if for every admissible semi-metric the scaling sequences are bounded.
See the survey \cite{VVZ23} for this approach.

To study dynamical systems with the discrete spectrum, in \cite{F51} Fomin introduced the concept of mean L-stability for topological dynamical systems.
In \cite{S82}, Scarpellini established some necessary and sufficient conditions for an $\mathbb{R}$-flow to have discrete spectrum. These conditions include Besicovitch almost periodic functions and B-stability with respect to a bounded measurable function.
In \cite{HLY11} Huang et al.\@ studied the measure-theoretical sensitivity and equicontinuity for an invariant measure,
and showed that measure-theoretical equicontinuity implies discrete spectrum.
In \cite{LTY15} Li et al.\@ introduced the concept of mean equicontinuity which is equivalent to mean L-stability, and showed that every ergodic measure on a mean equicontinuous system has discrete spectrum, answering an open question in \cite{S82}.
In \cite{G17} García-Ramos studied mean equicontinuity for invariant measures, and showed that
an ergodic measure has discrete spectrum if and only if it is mean equicontinuous. 
In \cite{GM19} García-Ramos and Marcus studied the mean equicontinuity with respect to square integrable functions,
and showed that for an ergodic measure, a square integrable function is mean equicontinuous if and only if it is almost periodic. 
In \cite{Y19} Yu further studied the measure-theoretic complexity with respect to a square integrable function,
and showed that a square integrable function is almost periodic if and only if it is mean equicontinuous if and only if the measure has bounded complexity with respect to this function.
See also \cite{AN23} and \cite{LSS24} for related characterizations of discrete spectrum by tameness of square integrable functions and mean almost periodic points respectively.
During their investigation of Sarnak’s conjecture,
Huang, Wang, and Ye \cite{HWY19} introduced the measure complexity of an invariant measure via the mean metric, proving that an invariant measure of a topological dynamical system has discrete spectrum if and only if the invariant measure has bounded measure complexity. 
Later in \cite{HLTXY21} Huang et al.\@ established the equivalence of discrete spectrum, bounded measure complexity and mean equicontinuity for an invariant measure. 
This result was subsequently extended to group actions. For countable discrete amenable group actions, in \cite{YZZ21} Yu et al.\@ proved that the equivalence holds along any F{\o}lner sequence.  
See \cite{XX24} for general discrete group actions, 
and \cite{HMXZ24} for locally compact group actions.

In this paper, we first introduce several types of low complexity for a collection of finite partitions in a standard probability space,  
without being restricted by the framework of a measure-preserving system. 
The following result is the first main theorem of this paper, which reveals the equivalence of precompactness, zero maximal pattern entropy and bounded mean complexity of a collection of finite partitions in a standard probability space.

\begin{thm}\label{thm:main1}
Let $(X, \mathcal{B}, \mu)$ be a standard probability space.
For $r\geq 2$, let $\mathfrak{P}_r$ be the collection of all partitions of $X$ with at most $r$ atoms.  
Then for a subset $K$ of $\mathfrak{P}_r$, the following assertions are equivalent:
\begin{enumerate}
    \item $K$ is precompact in the Rokhlin metric on $\mathfrak{P}_r$;
    \item the maximal pattern entropy of $K$ is zero;
    \item the set $\{\mathbf{1}_{A} \colon A \in \alpha \in K\}$ is precompact in $L^2(\mu)$;
    \item $K$ has bounded mean complexity with respect to the Hamming distance.
\end{enumerate} 
\end{thm} 

Next, we study the  complexity of a finite partition in a measure-preserving system with the acting group being a countably infinite discrete amenable group.
The following result is the second main theorem of this paper, which establishes the equivalence of several types of low complexity of partitions.
It provides a local version of the characterization of discrete spectrum for measure-preserving systems by low complexity,
which is new even for $\mathbb{Z}$-actions.

\begin{thm}\label{thm:main2}
Let $(X,\mathcal{B},\mu, G)$ be a measure-preserving system with  $G$ being a countably infinite discrete amenable group.
For a finite partition $\alpha$ of $X$, 
the following assertions are equivalent:
\begin{enumerate}
    \item the maximal pattern entropy of $\alpha$ is zero;
    \item for each $A\in \alpha$, $\mathbf{1}_{A}$ is  almost periodic;
    \item $\alpha$ has bounded mean complexity with respect to some (and hence any) F{\o}lner sequence in $G$;
    \item $\alpha$ is mean equicontinuous with respect to some (and hence any) tempered F{\o}lner sequence in $G$.
\end{enumerate}
\end{thm}

The structure of the paper is organized as follows. 
In Section 2, we introduce several types of low complexity
for a collection of finite partitions in a standard probability space, and show some basic properties.
Section 3 is devoted to proving Theorem~\ref{thm:main1}.
In Section 4, we first apply Theorem~\ref{thm:main1} to nonsingular actions and measure-preserving actions on standard probability spaces.
Then we study the complexity of a partition with respect to a F{\o}lner sequence in the amenable acting group, and prove Theorem~\ref{thm:main2}.
In Section 5, we study the complexity of an invariant measure in a topological dynamical system using the metric on the state space, and show that our results extend the corresponding results in the literature.
  
\section{Preliminaries on the complexity of finite partitions}

For a finite set $A$, denote by $|A|$ the number of elements of $A$.
A probability space  $(X, \mathcal{B}, \mu)$ is called \emph{standard} if it is isomorphic modulo zero to a   complete separable metric space equipped with the completion of a Borel probability measure. 
In this section,   
we introduce some concepts of low  complexity for a collection of finite partitions in a standard probability space and show some basic properties.

\subsection{Finite measurable partition and entropy}
Let $(X, \mathcal{B}, \mu)$ be a standard probability space. 
A finite \emph{partition} of $X$ is a finite collection $\alpha$ of disjoint measurable sets $A_i$ (called \emph{atoms} of $\alpha$) whose union covers $X \pmod 0$.
We say that a partition $\alpha$ is \emph{finer} than 
another partition $\beta$, denoted by $\alpha \succeq \beta$, if each atom in $\beta$ is a union of atoms of $\alpha \pmod 0$. 
Partitions $\alpha$ and $\beta$ are \emph{equivalent} if each is a refinement of the other. 
Denote by $\mathfrak{P}$ the collection of all finite measurable partitions of $X$. 
We will identify two partitions provided that they are equivalent, that is, we deal with equivalence classes of partitions.

For a partition $\alpha=\{A_1,A_2,\dotsc,A_r\}$ of $X$, we define the \emph{entropy} of $\alpha$ by 
\[
    H_{\mu}(\alpha)=-\sum_{i=1}^r \mu(A_i)\log \mu(A_i).
\]
Recall that we make the convention that $0\log 0=0$.

For measurable subsets $A,B\subset X$ with $\mu(B) > 0$, set $\mu(A|B)=\frac{\mu(A\cap B)}{\mu(B)}$.
If $\mu(B)=0$, it is convenient to make the convention that $\mu(A|B)=0$.
For two finite partitions $\alpha$ and $\beta$ of $X$,
the \emph{conditional entropy} of $\alpha$ with respect to $\beta$
is defined by the formula
\[
 H_{\mu}(\alpha|\beta)=
 \sum_{B \in \beta}\mu(B)H_{\mu}(\alpha|B),
 \]
 where
 \[ 
 H_{\mu}(\alpha|B) =\sum_{A \in \alpha}-\mu(A|B)\log \mu(A|B).
 \]
The quantity $H_{\mu}(\alpha|\beta)$ is the average entropy of the partition induced by $\alpha$
on an atom of $\beta$.

For $\alpha,\beta \in \mathfrak{P}$, define
\[
\rho(\alpha,\beta)=H_{\mu}(\alpha|\beta)+H_{\mu}(\beta|\alpha).
\]
It is easy to see that the function $\rho$ defines a metric on $\mathfrak{P}$, 
which is called the \emph{Rokhlin metric}.
It is necessary to consider the collection of finite partitions with at most a given number of atoms.
For a fixed integer $r\geq 2$,
denote by $\mathfrak{P}_r$ the collection of partitions in $\mathfrak{P}$ with at most $r$ atoms. 

For $\alpha,\beta \in \mathfrak{P}_r$, 
write $\alpha=\{A_1,A_2,\dotsc,A_r\}$ and $\beta=\{B_1,B_2,\dotsc,B_r\}$, by adding empty sets if necessary.
Define
\[
\widetilde{\rho} (\alpha,\beta)=\min_{\sigma \in S_r }\sum_{i=1}^r \mu(A_i \Delta B_{\sigma(i)}),
\]
where $S_r$ is the set of all permutations of $\{1,\cdots,r\}$. 
It is easy to check that $\widetilde{\rho}$ is a metric on $\mathfrak{P}_r$. 
We need the following folklore lemma, see e.g. \cite[Facts 1.7.7 and 1.7.15]{D11}.

\begin{lem}\label{lem:Rokhlin-metric-eq}
Let $(X,\mathcal{B},\mu)$ be a standard probability space and $r\geq 2$ be an integer. 
Then 
\begin{enumerate}
\item the metric space $(\mathfrak{P}_r,\rho)$ is complete and separable;
\item the metrics $\rho$ and $\widetilde{\rho}$ 
are uniformly equivalent on $\mathfrak{P}_r$.
\end{enumerate}
\end{lem}

We need the following technical result. 
Note that the case $r=2$ is proved in \cite{K67}.
\begin{lem}\label{lem:H-mu-rho}
Let $(X, \mathcal{B}, \mu)$ be a standard probability space
and $r\geq 2$ be a fixed integer.  
For any $\varepsilon>0$ there exists $\delta>0$ such that for any $\alpha \in \mathfrak{P}_r$ and $\beta\in\mathfrak{P}$, 
if $H_{\mu}(\alpha|\beta)<\delta$, 
then there exists $\beta ' \preceq \beta $ with $\beta '\in \mathfrak{P}_{r+1}$ such that $\widetilde{\rho}(\alpha,\beta')<\varepsilon$.
\end{lem}

\begin{proof}
Fix $\varepsilon>0$, $\alpha \in \mathfrak{P}_r$ and $\beta\in\mathfrak{P}$.
Enumerate $\alpha=\{A_1,\dotsc,A_r\}$ and $\beta=\{B_1,\dotsc,B_m\}$. 
Pick a sufficiently small constant $a$ such that  $0<a<\min\{\frac{1}{r},\frac{\varepsilon}{2r^2},\frac{1}{e}\}$ and $-a\log a < \frac{\log2}{2}$.
Let $C>\frac{2(r+1)}{\varepsilon}$ and $0<\delta<-\frac{a\log a}{C} $. 
Suppose that $H_{\mu}(\alpha|\beta)<\delta$.
Define 
    \[
    F=\{j\in \{1,\dotsc,m\}\colon 
    H_{\mu}(\alpha|B_j)<C\delta <-a\log a\}.
    \]
Since
    \[
     \delta>H_{\mu}(\alpha|\beta)=\sum_{j=1}^m \mu(B_j)H_{\mu}(\alpha|B_j)\geq \mu\biggl(\bigcup_{j \in \{1,\dotsc,m\}\setminus F} B_j\biggr) C\delta,
     \] 
$\mu(\bigcup_{j \in F}B_j) \geq 1-\frac{1}{C}$. 
 Let $b>\max\{\frac{1}{e},\frac{1}{2}\}$  with $-b \log b=-a\log a$. 
For any $i \in \{1,\dotsc,r\}$ and $j \in F$,
\[
    H_\mu(\alpha|B_j)=\sum_{i=1}^r -\mu(A_i|B_j)\log \mu(A_i|B_j)<C\delta,
\]
then either $\mu(A_i | B_j)>b$ or $\mu(A_i | B_j)<a$.
For any $i \in \{1,\dotsc,r\}$, let
\[
    E_i=\{j \in F \colon \mu(A_i | B_j)>b\}\text{ and }D_i=\bigcup_{j \in E_i} B_j.
\] 
Finally, let $D_{r+1}=\bigcup_{j \in \{1,\dotsc,m\} \setminus F} B_j$. Then $\mu(D_{r+1})\leq \frac{1}{C}$.

We claim that $\{E_i\colon 1\leq i\leq r\}$ is a partition of $F$.
In fact, 
if there exists $j \in F \setminus\bigcup_{i=1}^r E_i$, 
then for any $i \in \{1,\dotsc,r\}$, 
we have $\mu(A_i | B_j)<a$. 
 Thus
    \[
    \mu(B_j)=\sum_{i=1}^r \mu(A_i \cap B_j) \leq ra \mu(B_j)<\mu(B_j),
    \]
which is a contradiction. 
Assume that there exist $1\leq i_1<i_2\leq r$ such that $E_{i_1}\cap E_{i_2}\neq\emptyset$. 
Pick $j\in E_{i_1}\cap E_{i_2}$.
Then $\mu(A_{i_1}|B_j)>b>\frac{1}{2}$ and $\mu(A_{i_2}|B_j)>b>\frac{1}{2}$,
which implies that $\mu(A_{i_1}\cap B_j)> \frac{1}{2}\mu(B_j)$ and $\mu(A_{i_2}\cap B_j)> \frac{1}{2}\mu(B_j)$.
This contradicts the fact that $A_{i_1}\cap   A_{i_2}=\emptyset$. 
    
Let $\beta'=\{D_1,\dotsc,D_r,D_{r+1}\}$.
Then $\beta' \preceq \beta$ and $\beta'\in\mathfrak{P}_{r+1}$. 
For any $i \in \{1,\dotsc,r\}$ and $j \in E_i$, we have $\mu(A_i | B_j)>b$ and  $\ell \in \{1,\dotsc,r\} \setminus \{i\}$, $\mu(A_\ell | B_j)<a$. 
Thus for any $i \in \{1,\dotsc,r\}$, we have
\begin{align*}
\mu(A_i \triangle D_i)
    &=\mu(D_i \setminus A_i)+\mu(A_i \setminus D_i) 
    =\mu\biggl(\bigcup_{j \in E_i}B_j  \cap A_i^c \biggr) +\mu(A_i \cap D_i^c) \\
    &= \sum_{j \in E_i} \mu(B_j \cap A_i^c)+
       \sum_{j \in \{1,\dotsc,r\} \setminus \{i\}} \mu(A_i \cap D_j) + \mu(A_i \cap D_{r+1}) \\
    &\leq \sum_{j \in E_i } \sum_{\ell\in \{1,\dotsc,r\} \setminus \{i\}} \mu(B_j \cap A_\ell) + a \sum_{j \in \{1,\dotsc,r\} \setminus \{i\}} \mu(D_j)+ \mu(D_{r+1}) \\
    &\leq \sum_{j \in E_i}a(r-1)\mu(B_j)+a +\tfrac{1}{C}\\
    &\leq a(r-1) +a+\tfrac{1}{C}=ar+\tfrac{1}{C}.
    \end{align*}
Thus
\[ 
\sum_{i=1}^r \mu(A_i \triangle D_i)+\mu(\emptyset \triangle D_{r+1}) 
\leq \sum_{i=1}^r \Bigl(ar +\tfrac{1}{C}\Bigr) +\tfrac{1}{C} \leq ar^2 +
    \tfrac{r+1}{C} \leq \varepsilon.
\]
This implies that $\widetilde{\rho}(\alpha,\beta')<\varepsilon$.
\end{proof}

\subsection{Sequence entropy of partitions}
Let $(X, \mathcal{B}, \mu)$ be a standard probability space.
For $\alpha,\beta \in \mathfrak{P}$, define $\alpha \vee \beta = \{A_i \cap B_j : A_i \in \alpha, B_j \in \beta\}$. 
For a sequence $\mathcal{A}=(\alpha_{i})_{i=1}^{\infty}$ in $\mathfrak{P}$,
we define the \emph{sequence entropy} of  $(\alpha_{i})_{i=1}^{\infty}$ by 
\[
h_{\mu}(\mathcal{A})=
\limsup_{n \to \infty} \frac{1}{n} H_{\mu}\biggl(\bigvee_{i=1}^n \alpha_{i}\biggr).
\]

Following the idea in \cite{HY09}, we introduce the concept of maximal pattern entropy of a collection of partitions.
For a subset $K$ of $\mathfrak{P}$, 
we define the \emph{maximal pattern entropy} of $K$ by
\[
h^{*}_{\mu}(K)= 
\limsup_{n \to \infty} \frac{1}{n} P^{*}_{K,\mu}(n),  
\]
where for each $n\in\mathbb{N}$,
\[
P^{*}_{K,\mu}(n)=\sup_{\alpha_1,\dotsc,\alpha_n \in K} H_{\mu}\biggl(\bigvee_{i=1}^n \alpha_i\biggr).
\]
It is easy to see that the sequence $(P^{*}_{K,\mu}(n))$ is subadditive; hence, by Fekete's lemma, one has
\[
h^{*}_{\mu}(K)= 
\lim_{n \to \infty} \frac{1}{n} P^{*}_{K,\mu}(n). 
\]

\begin{lem}\label{lem:max-pattern-entropy-eq}
Let $(X, \mathcal{B}, \mu)$ be a standard probability space and $K$ be a nonempty subset of $\mathfrak{P}_r$ with $r\geq 2$.
Then 
\[
    h^{*}_{\mu}(K)= \sup_{\mathcal{A}\in \mathcal{S}_K } h_{\mu}(\mathcal{A}),
\]
where $\mathcal{S}_K$ is the collection of all sequences $\mathcal{A}=(\alpha_{i})_{i=1}^{\infty}$ of $K$,
and there is some $\mathcal{A} \in \mathcal{S}_K $ such that $h^{*}_{\mu}(K)= h_{\mu}(\mathcal{A})$. 
\end{lem}

\begin{proof}
We can assume that $K$ is infinite, otherwise it is clear that $h^{*}_{\mu}(K)=0$.
For any $\mathcal{A}=(\alpha_{i})_{i=1}^{\infty} \in \mathcal{S}_K$,
it is clear that $ P^{*}_{K,\mu}(n) \geq H_{\mu}(\bigvee_{i=1}^n \alpha_{i})$. 
Then $h^{*}_{\mu}(K) \geq h_{\mu}(\mathcal{A})$.
This implies $h^{*}_{\mu}(K) \geq \sup_{\mathcal{A} \in \mathcal{S}_K } h_{\mu}(\mathcal{A})$.

For each $n\in\mathbb{N}$, pick a subset $D_n$ of $K$ such that 
\[
\frac{1}{|D_n|}H_\mu \biggl(\bigvee_{\alpha\in D_n}\alpha\biggr)
    >h_\mu^*(K)-\frac{1}{n}.
\]
We can require that the number of partitions in $D_n$ is large enough. 
So without loss of generality, we assume that 
\[
    \frac{|D_n|}{|D_1|+\dotsb+|D_n| }>1-\frac{1}{n}. 
\]
Let $\mathcal{A}$ be the sequence of all partitions in those $D_n$. 
To be precise, enumerate each $D_n$ as $\{\alpha_{1}^{(n)},\dotsc,\alpha_{|D_n|}^{(n)} \}$ and 
let 
\[
\mathcal{A}=(\alpha_{1}^{(1)},\dotsc,\alpha_{|D_1|}^{(1)},\alpha_{1}^{(2)},\dotsc,\alpha_{|D_2|}^{(2)},\dotsc,\alpha_{1}^{(n)},\dotsc,\alpha_{|D_n|}^{(n)},\dotsc).
\]
Then 
\begin{align*}
 h_{\mu}(\mathcal{A})&\geq \limsup_{n\to\infty}
 \frac{1}{|D_1|+\dotsb+|D_n|} H_\mu\biggl(\bigvee_{i=1}^n \bigvee_{\alpha\in D_i} \alpha\biggr)\\
 &\geq 
 \limsup_{n\to\infty}
 \frac{|D_n|}{|D_1|+\dotsb+|D_n|}
 \frac{1}{|D_n|} H_\mu \biggl(\bigvee_{\alpha\in D_n} \alpha\biggr)\\
 &\geq   \limsup_{n\to\infty}\Bigl(1-\frac{1}{n}\Bigr)
 \Bigl(h_\mu^*(K)-\frac{1}{n}\Bigr)\\
 &= h_\mu^*(K).
\end{align*}
This ends the proof.
\end{proof}

\subsection{Bounded mean complexity of partitions}
Let $(X, \mathcal{B}, \mu)$ be a standard probability space and $r\geq 2$ be an integer.
Let $E$ be a finite subset of $\mathfrak{P}_r$.
For each $\alpha\in E$, enumerate $\alpha$ as $\{A_1^\alpha,A_2^\alpha,\dotsc,A_r^\alpha\}$.
For every $x\in X$, we can define the \emph{$E$-name} $E_x$ of $x$ as a sequence in $\{1,2,\dotsc,r\}^E$ where
$E_x(\alpha)=i$ whenever $x\in A^\alpha_i$.
The \emph{normalized Hamming distance} of two points $x,y\in X$ with respect to $E$ is defined by
\[
H_E(x,y)=\frac{1}{|E|} |\{\alpha\in E\colon E_x(\alpha)\neq E_y(\alpha)\}|. 
\]
It is easy to check that 
\[
    H_E(x,y)= \frac{1}{2|E|} \sum_{\alpha\in E} \sum_{A \in \alpha} |\mathbf{1}_{A}(x)-\mathbf{1}_{A}(y)|.
\]
If $E=\{\alpha\}$, then we use the symbol $H_\alpha(x,y)$ instead of $H_{\{\alpha\}}(x,y)$.
In general, one has $H_E(x,y)=\frac{1}{|E|}\sum_{\alpha\in E}H_\alpha(x,y)$.
Notice that 
$H_\alpha(x,y)=0$ or $1$.

For $x\in X$ and $\varepsilon>0$,
denote
\[
B_{H_E}(x,\varepsilon)=\{y \in X\colon  H_E(x,y)<\varepsilon \}
\]
and 
\[
\mathcal{C}(E,\varepsilon)=\min\biggl\{|D|\colon 
D\subset X\text{ s.t. }\mu\biggl(\bigcup_{x\in D} B_{H_E}(x,\varepsilon)\biggr)>1-\varepsilon\biggr\}.
\]

We say that a subset $K$ of $\mathfrak{P}_r$ has \emph{bounded mean complexity}
if for any $\varepsilon>0$ there exists a constant $C=C(\varepsilon)\in \mathbb{N}$ such that $\mathcal{C}(E,\varepsilon)\leq C$
for any finite subset $E$ of $K$.
For a sequence $\mathcal{A}=(\alpha_{i})_{i=1}^{\infty}$ in $\mathfrak{P}_r$, 
we say that $\mathcal{A}$ has \emph{bounded mean  complexity}
if for any $\varepsilon>0$ there exists a constant $C=C(\varepsilon)\in \mathbb{N}$ such that $\mathcal{C}(\{\alpha_1,\dotsc,\alpha_n\},\varepsilon) \leq C$
for all $n\in\mathbb{N}$.

\begin{rem}
For a finite subset $E$ of $\mathfrak{P}_r$, it is easy to check that for any $\varepsilon>0$,
\[
\mathcal{C}(E,\varepsilon)\leq \biggl| \bigvee_{\alpha\in E}\alpha\biggr|\leq r^{|E|}.
\]
Then a subset $K$ of $\mathfrak{P}_r$ has bounded mean  complexity 
if and only if for any $\varepsilon>0$ there exists a constant $C\in \mathbb{N}$ and $N\in\mathbb{N}$ such that $\mathcal{C}(E,\varepsilon)\leq C$
for any finite subset $E$ of $K$ with $|E|\geq N$.
And a sequence $\mathcal{A}=(\alpha_{i})_{i=1}^{\infty}$ in $\mathfrak{P}_r$,  has bounded mean  complexity
if and only if for any $\varepsilon>0$ there exists a constant $C\in \mathbb{N}$ and $N\in\mathbb{N}$ such that $\mathcal{C}(\{\alpha_1,\dotsc,\alpha_n\},\varepsilon) \leq C$
for all $n\geq N$.
\end{rem}

\begin{rem}
If a subset $K$ of $\mathfrak{P}_r$ has bounded mean  complexity, then it is clear that any sequence in $K$ also has bounded mean  complexity.
It will be shown in Proposition~\ref{prop:precpt-bd-mean-complexity} that the converse also holds.
\end{rem} 

Let $(E_n)_{n=1}^\infty$ be a sequence of finite subsets of $\mathfrak{P}_r$.
We say that $(E_n)_{n=1}^\infty$ is \emph{mean equicontinuous}
if for any $\varepsilon>0$, there exist 
pairwise disjoint measurable subsets $B_1,B_2,\dotsc,B_k$
with $\mu(\bigcup_{i=1}^k B_i)>1-\varepsilon$ such that 
for every $x,y\in B_i$ for some $i\in\{1,2,\dotsc,k\}$, one has 
\[  
\limsup_{n \to \infty}   H_{E_n} (x,y) <\varepsilon.
\] 

\begin{lem}\label{lem:partition-seq-mean-equi}
Let $(X, \mathcal{B}, \mu)$ be a standard probability space
and $r\geq 2$ be a fixed integer.  
If  a sequence $(E_n)_{n=1}^\infty$ of finite subsets of $\mathfrak{P}_r$ is  mean equicontinuous, then 
for any $\varepsilon>0$, there exist 
pairwise disjoint measurable subsets $B_1,B_2,\dotsc,B_k$
with $\mu(\bigcup_{i=1}^k B_i)>1-\varepsilon$ such that 
for every $x,y\in B_i$ for some $i\in\{1,2,\dotsc,k\}$, one has $H_{E_n} (x,y) <\varepsilon$ for all $n\in\mathbb{N}$. 
\end{lem}

\begin{proof}
Fix an arbitrary $\varepsilon>0$.
As $(E_n)_{n=1}^\infty$ is mean equicontinuous, there exist 
pairwise disjoint measurable subsets $B_1,B_2,\dotsc,B_k$
with $\mu(\bigcup_{i=1}^k B_i)>1-\frac{\varepsilon}{4}$ such that 
for every $x,y\in B_i$ for some $i\in\{1,2,\dotsc,k\}$, one has 
\[  
\limsup_{n \to \infty}   H_{E_n} (x,y) <\tfrac{\varepsilon}{4}.
\] 
For each $i\in\{1,\dotsc,k\}$, pick $x_i\in B_i$ and for each $N\in\mathbb{N}$, let
\[
    U_N(x_i)=\{y\in B_i\colon H_{E_n}(x_i,y)<\tfrac{\varepsilon}{2},\ \forall n\geq N\}.
\]
It is clear that for each $i\in\{1,\dotsc,k\}$, $U_N(x_i)\subset U_{N+1}(x_i)$ and 
$B_i=\bigcup_{N=1}^\infty U_N(x_i)$.
We pick $N_0\in\mathbb{N}$ such that $\mu(B_i\setminus  U_{N_0}(x_i))<\frac{\varepsilon}{4k}$ for $i\in\{1,\dotsc,k\}$.
Then $\mu(\bigcup_{i=1}^k U_{N_0}(x_i))>1-\frac{\varepsilon}{2}$.
Let $\alpha$ be the partition $ \{ U_{N_0}(x_1),U_{N_0}(x_2),\dotsc,U_{N_0}(x_k), X\setminus \bigcup_{i=1}^k U_{N_0}(x_i)\}$ of $X$.
Enumerate the nonempty atoms in the partition $\alpha\vee \bigvee_{n=1}^{N_0-1}\bigvee_{\beta \in E_n}\beta$ which are contained in $\bigcup_{i=1}^k U_{N_0}(x_i)$ as $A_1,A_2,\dotsc,A_m$.
Then $\mu(\bigcup_{i=1}^m A_i) =\mu(\bigcup_{i=1}^k U_{N_0}(x_i))>1-\frac{\varepsilon}{2}$.
For any $x,y\in A_i$ for some $i\in\{1,2,\dotsc,m\}$, and for any $n\in\mathbb{N}$,
if $n\leq N_0-1$, then $A_i$ is contained in an atom of $\bigvee_{\beta\in E_n}\beta$, and thus one has $H_{E_n}(x,y)=0$;
if $n\geq N_0$, then there exists $j\in\{1,2,\dotsc,k\}$ such that $A_i\subset U_{N_0}(x_j)$, and thus
$H_{E_n}(x,y)\leq H_{E_n}(x,x_j)+H_{E_n}(x_j,y)<\varepsilon$.
This ends the proof.
\end{proof}

\begin{rem}
Let $(\alpha_{i})_{i=1}^{\infty}$ be a sequence  in $\mathfrak{P}_r$.
For each $n\in\mathbb{N}$, let $E_n=\{\alpha_1,\dotsc,\alpha_n\}$.
If the sequence $(E_n)_{n=1}^\infty$ is mean equicontinuous,
then by Lemma~\ref{lem:partition-seq-mean-equi}, the sequence $(\alpha_{i})_{i=1}^{\infty}$ has bounded mean complexity.
It is of interest to know when the converse holds.
It will be shown in Proposition~\ref{prop:mps-mean-eq} that
a partition along a tempered F{\o}lner sequence of the acting amenable group is mean equicontinuous if and only if it has bounded mean complexity.
\end{rem}

\section{Proof of Theorem~\ref{thm:main1}}

The aim of this section is to prove Theorem~\ref{thm:main1}.
We split the proof into several parts as follows.
First we prove the equivalence of precompactness in the Rokhlin metric and zero maximal pattern entropy for a collection of finite partitions.

\begin{prop}\label{prop:precpt-entropy}
Let $(X, \mathcal{B}, \mu)$ be a standard probability space and $r\geq 2$ be an integer.
Then a nonempty subset $K$ of $\mathfrak{P}_r$ is precompact in $(\mathfrak{P}_r,\rho)$
if and only if the maximal pattern entropy of $K$ is zero.
\end{prop}

\begin{proof}
($\Rightarrow$)
Fix an arbitrary $\varepsilon>0$.
Since $K$ is precompact in $(\mathfrak{P}_r,\rho)$, 
there exists a finite subset $\{\beta_1,\dotsc,\beta_m\} $ of $K$ such that for any $\alpha \in K$ there exists $j \in \{1,\dotsc,m\}$ with $\rho(\alpha,\beta_j)<\frac{\varepsilon}{2}$. 
Fix an integer $n$ with 
$\frac{1}{n} \sum_{j=1}^m H_{\mu}(\beta_j)<\frac{\varepsilon}{2}$ 
and choose $\alpha_1,\alpha_2,\dotsc,\alpha_n \in K$.
For each $i=1,2,\dotsc,n$, 
there exists $j_i \in \{1,\dotsc,m\}$ such that $\rho(\alpha_i,\beta_{j_i})<\frac{\varepsilon}{2}$.
For each $k=1,2,\dotsc,m$, let $D_k=\{1\leq i \leq n\colon j_i=k\}$.
It is clear that $\{D_1,\dotsc,D_m\}$ is a partition of $\{1,\dotsc,n\}$. 
Then
\begin{align*}      
    \frac{1}{n}H_{\mu} \biggl(\bigvee_{i=1}^n \alpha_i\biggr) &\leq 
    \frac{1}{n}H_{\mu}\biggl( \bigvee_{j=1}^m \beta_j  \vee \bigvee_{i=1}^n \alpha_i\biggr) 
        \leq \frac{1}{n}\sum_{j=1}^m H_{\mu}\biggl(\beta_j \vee   \bigvee_{i\in D_j} \alpha_i\biggr) \\
    &= \frac{1}{n}\sum_{j=1}^m \biggl(H_{\mu}(\beta_j) + H_{\mu}\biggl( \bigvee_{i\in D_j} \alpha_i \biggl| \beta_j\biggr)\biggr)
    \leq \frac{1}{n} \sum_{j=1}^m \biggl(H_{\mu}(\beta_j) + \sum_{i \in D_j}H_{\mu}(\alpha_i|\beta_j)\biggr) \\
    &\leq \frac{1}{n}\biggl(\sum_{j=1}^m H_{\mu}(\beta_j) + n\frac{\varepsilon}{2}\biggr)<\varepsilon.
    \end{align*}
This implies that $\frac{1}{n} P^{*}_{K,\mu}(n)<\varepsilon$ for all 
$n > \sum_{j=1}^m H_{\mu}(\beta_j)\cdot \frac{2}{\varepsilon}$.
By the arbitrariness of $\varepsilon$, we have  $h^{*}_{\mu}(K)=0$.

($\Leftarrow$)
Assume that $K$ is not precompact in $(\mathfrak{P}_r,\rho)$. By Lemma~\ref{lem:Rokhlin-metric-eq}, $K$ is also not precompact in $(\mathfrak{P}_r,\widetilde{\rho})$.
Then there exist $\varepsilon>0$ and a sequence $(\alpha_n)_{n=1}^{\infty}$ in $K$ such that for any $m \neq n$, $\widetilde{\rho}(\alpha_m,\alpha_n)>\varepsilon$.  
Let $\delta$ be the constant in Lemma~\ref{lem:H-mu-rho} for the above $r$ and $\frac{\varepsilon}{2}$. 
    
We claim that there exists a subsequence $(\alpha_{n_i})_{i=1}^{\infty}$ of $(\alpha_n)_{n=1}^{\infty}$ such that for any $i \geq 2$, we have $H_{\mu}(\alpha_{n_i}|\bigvee_{j=1}^{i-1}\alpha_{n_j})\geq\delta$. 
Let $\alpha_{n_1}=\alpha_1$ and assume that $\alpha_{n_1},\dotsc,\alpha_{n_{k-1}}$ have been chosen. 
By Lemma~\ref{lem:H-mu-rho} for any $\beta \in \mathfrak{P}_r$ with $H_{\mu}(\beta|\bigvee_{j=1}^{k-1}\alpha_{n_j})<\delta$, 
there exists $\alpha \preceq \bigvee_{j=1}^{k-1}\alpha_{n_j}$ such that $\widetilde{\rho}(\beta,\alpha)<\frac{\varepsilon}{2}$. 
Since there are only finitely many partitions $\beta$ such that $\beta\preceq \bigvee_{j=1}^{k-1}\alpha_{n_j}$,
the set $\{i\in\mathbb{N} \colon H_{\mu}(\alpha_i|\bigvee_{j=1}^{k-1}\alpha_{n_j})<\delta \}$ is finite.
Thus there exists $n_k>n_{k-1}$ such that $H_{\mu}(\alpha_{n_k}|\bigvee_{j=1}^{k-1}\alpha_{n_j})\geq\delta$. 
By induction, we obtain the desired sequence.

As the maximal pattern entropy of $K$ is zero, by Lemma~\ref{lem:max-pattern-entropy-eq} every sequence in $K$ also has zero sequence entropy.
But  
\[
    \limsup_{k \to \infty} \frac{1}{k}H_{\mu}\biggl(\bigvee_{j=1}^k \alpha_{n_j}\biggr)=\limsup_{k \to \infty}  \frac{1}{k} \biggl(\sum_{j=2}^k H_{\mu}\biggl(\alpha_{n_j}\biggl| \bigvee_{\ell=1}^{j-1}\alpha_{n_\ell}\biggr)+H_{\mu}(\alpha_{n_1})\biggr) \geq \delta,
\]
which is a contradiction.
Hence, $K$ is precompact.
\end{proof} 

Next, we show the equivalence of precompactness of a collection of partitions and the collection of characteristic functions of atoms in partitions.
In this case, it is convenient to consider the uniformly equivalent metric $\widetilde{\rho}$ on $\mathfrak{P}_r$.

\begin{prop}\label{prop:precpt-L2}
Let $(X, \mathcal{B}, \mu)$ be a standard probability space and $r\geq 2$ be an integer.
Then a subset $K$ of $\mathfrak{P}_r$ is precompact in $(\mathfrak{P}_r,\widetilde{\rho})$
if and only if the set $\{\mathbf{1}_{A} \colon A \in \alpha \in K\}$ is precompact in $L^2(\mu)$. 
\end{prop}

\begin{proof}
We first prove the case $r=2$.

($\Rightarrow$)
For any $\varepsilon>0$, we can find a finite subset $E$ of $K$ such that for any $\alpha\in K$, there exists $\beta\in E$ such that $\widetilde{\rho}( \alpha,\beta)<\varepsilon^2$.
Now fix $\alpha\in K$ and $A\in \alpha$.
There exists $\beta \in E$ such that $\widetilde{\rho}( \alpha,\beta)<\varepsilon^2$. 
Then there exists $B\in\beta$ such that 
$\mu(A\Delta B)\leq \widetilde{\rho}( \alpha,\beta)<\varepsilon^2$. 
Note that 
\[
\Vert \mathbf{1}_A-\mathbf{1}_B\Vert_{L^2(\mu)}^2=\int |\mathbf{1}_A-\mathbf{1}_B|^2 \dd\mu=\int |\mathbf{1}_A-\mathbf{1}_B| \dd\mu=\mu(A\Delta B).
\]
Then $\Vert \mathbf{1}_A-\mathbf{1}_B\Vert_{L^2(\mu)}<\varepsilon$.
This implies that $\{\mathbf{1}_B\colon B\in \beta\in E\}$ is a  finite $\varepsilon$-net of  $\{\mathbf{1}_{A} \colon A \in \alpha \in K\}$. 
Thus  $\{\mathbf{1}_{A} \colon A \in \alpha \in K\}$ is precompact  in $L^2(\mu)$.

($\Leftarrow$) 
For any $\varepsilon>0$, 
there exists a finite set 
$\{\mathbf{1}_{A_1},\dotsc,\mathbf{1}_{A_n}\}$ such that for any $\mathbf{1}_A \in \{\mathbf{1}_A \colon A \in\alpha \in K\}$, 
there exists $i \in \{1,\dotsc,n\}$ such that 
$\|\mathbf{1}_A-\mathbf{1}_{A_i}\|_{L^2(\mu)}<
\sqrt{\frac{\varepsilon}{2}}$. 
For any $\beta=\{B,B^c\}\in K$,
there exists $i\in \{1,\dotsc,n\}$ such that 
$\|\mathbf{1}_B-\mathbf{1}_{A_i}\|_{L^2(\mu)}<\sqrt{\frac{\varepsilon}{2}}$. 
Then $\mu(B\Delta A_i)<\frac{\varepsilon}{2}$, and 
\[
    \widetilde{\rho}(\beta,\{A_i,A_i^c\})
    =\mu(B\Delta A_i)+\mu(B^c \Delta A_i^c)<\varepsilon.
\]
This implies that 
$\{\{A_i,A_i^c\}\colon 1\leq i\leq n\}$ is a finite $\varepsilon$-net of $K$. 
Thus, $K$ is precompact. 

Now we can consider the general case for $r\geq 2$. 
Let $ K_2=\{\{A,A^c\}\colon A\in \alpha\in K\}$. 

($\Rightarrow$) 
As $K$ is precompact in $(\mathfrak{P}_r,\widetilde{\rho})$, by Lemma~\ref{lem:Rokhlin-metric-eq}, $K$ is also precompact in $(\mathfrak{P}_r,\rho)$. 
Then by Proposition~\ref{prop:precpt-entropy}, $h^*_\mu(K)=0$.
By Lemma~\ref{lem:max-pattern-entropy-eq}, there exists $\mathcal{A}'=(\{A_i,A_i^c\})_{i=1}^{\infty}$ in $K_2$ such that $h_{\mu}(\mathcal{A}')= h^{*}_{\mu}(K_2)$. 
For any $i \in \bbn$, there is $\alpha_i\in K$ such that $A_i \in \alpha_i$. 
Then $\mathcal{A}=(\alpha_i)_{i=1}^{\infty}$ is a sequence in $K$.
Since $\{A_i,A_i^c\} \preceq \alpha_i$, one has
\[
\frac{1}{n} H_{\mu}\biggl(\bigvee_{i=1}^n \{A_i,A_i^c\}  \biggr) \leq \frac{1}{n} H_{\mu}\biggl(\bigvee_{i=1}^n \alpha_{i}\biggr).
\]
Then 
\[
h^{*}_{\mu}(K_2)=h_{\mu}(\mathcal{A}') \leq h_{\mu}(\mathcal{A}) \leq h^{*}_{\mu}(K)=0. 
\]
By Lemma~\ref{lem:Rokhlin-metric-eq} and Proposition~\ref{prop:precpt-entropy},  $K_2$ is precompact in $(\mathfrak{P}_2,\widetilde{\rho})$. 
As the result holds for $r=2$, the set $\{\mathbf{1}_{A} \colon A \in \alpha \in K_2\}$ is precompact  in $L^2(\mu)$. Then $\{\mathbf{1}_{A} \colon A \in \alpha \in K\}$ is also precompact in $L^2(\mu)$, because 
$\{\mathbf{1}_{A} \colon A \in \alpha \in K\}\subset \{\mathbf{1}_{A} \colon A \in \alpha \in K_2\}$.

($\Leftarrow$)
Since $\{\mathbf{1}_{A} \colon A \in \alpha \in K\}$ is precompact in $L^2(\mu)$, $\{\mathbf{1}_{A^c} \colon A \in \alpha \in K\}$ is precompact in $L^2(\mu)$.
Then $\{\mathbf{1}_{A} \colon A \in \alpha \in K_2\}=\{\mathbf{1}_{A} \colon A \in \alpha \in K\}\cup \{\mathbf{1}_{A^c} \colon A \in \alpha \in K\}$
is precompact in $L^2(\mu)$.
As the result holds for $r=2$, $K_2$ is precompact  in $(\mathfrak{P}_2,\widetilde{\rho})$. 
By Lemma~\ref{lem:Rokhlin-metric-eq} and Proposition~\ref{prop:precpt-entropy}, $h^*_{\mu}(K_2)=0$.
For any finite subset $E$ of $K$, 
\[
    \bigvee_{\alpha\in E}\alpha = \bigvee_{A\in\alpha\in E}\{A,A^c\}.
\]
Then
\begin{align*}
\frac{1}{|E|}H_{\mu}\biggl(\bigvee_{\alpha\in E}\alpha\biggr) 
= \frac{1}{|E|}H_{\mu}\biggl( \bigvee_{A\in\alpha\in E}\{A,A^c\}\biggr)\leq \frac{1}{|E|} P_{K_2,\mu}^*(r|E|),
\end{align*}
which implies that $h^{*}_{\mu}(K)=0$ because $h^*_{\mu}(K_2)=0$.
Now by Proposition~\ref{prop:precpt-entropy} and Lemma~\ref{lem:Rokhlin-metric-eq},
$K$ is  precompact in $(\mathfrak{P}_r,\widetilde{\rho})$.
\end{proof}

\begin{coro}\label{coro:cmpact-L2}
Let $(X, \mathcal{B}, \mu)$ be a standard probability space.
Then a subset $K$ of $\mathfrak{P}_2$ is compact in $(\mathfrak{P}_2,\widetilde{\rho})$
if and only if the set $\{\mathbf{1}_{A} \colon A \in \alpha \in K\}$ is compact in $L^2(\mu)$. 
\end{coro}

\begin{proof}
($\Leftarrow$) 
By Proposition~\ref{prop:precpt-L2}, $K$ is precompact in $(\mathfrak{P}_2,\widetilde{\rho})$.
We only need to show that $K^c$ is open.
Fix $ \beta \in K^c$. There exists $B\in \beta$
such that $\mathbf{1}_{B} \notin \{\mathbf{1}_{A} \colon A \in \alpha \in K\}$. 
Since $\{\mathbf{1}_{A} \colon A \in \alpha \in K\}$ is closed in $L^2(\mu)$, there exists $\delta >0$ such that for any  $A \in \alpha \in K$, 
we have $\mu(A\Delta B)=\|\mathbf{1}_A-\mathbf{1}_{B}\|_{L^2(\mu)}^2\geq \delta$.
Then for any $\alpha\in  K$,
$ \widetilde{\rho}(\alpha,\beta) 
\geq \min_{A\in\alpha}\mu(A\Delta B)\geq \delta$.
This implies that $\beta$ is an interior point in $K^c$,
and then $K^c$ is open.

($\Rightarrow$)
By Proposition~\ref{prop:precpt-L2}, $\{\mathbf{1}_{A} \colon A \in \alpha \in K\}$  is precompact in $L^2(\mu)$.
Next we show that it is closed.
Assume that $(\mathbf{1}_{A_n})_{n=1}^\infty$ is a sequence in $\{\mathbf{1}_{A} \colon A\in \alpha \in K\}$ which converges to $f$ in $L^2(\mu)$ as $n\to\infty$. 
Then $f=\mathbf{1}_{A_0}$ $\mu$-a.e for some $A_0\in\mathcal{B}$. 
For each $n\in \mathbb{N}$, pick $\alpha_n\in K$ with $A_n\in \alpha_n$.
Since $K$ is compact in $(\mathfrak{P}_2,\widetilde{\rho})$, without loss of generality, 
there exists $\beta\in K$ such that  $\alpha_n$ converges to $\beta$ in $(\mathfrak{P}_2,\widetilde{\rho})$ as $n\to\infty$. 
For every $n\in \mathbb{N}$, there exists $B_n\in\beta$
such that $\mu(A_n\Delta B_n)\leq \widetilde{\rho}(\alpha_n,\beta)$.
As $\beta$ is a finite partition, without loss of generality, assume that $B_n=B$ for all $n\in \mathbb{N}$.
Then $\mu(A_n\Delta B)\to 0$ as $n\to\infty$.
As $\mathbf{1}_{A_n}$ converges to $\mathbf{1}_{A_0}$ in $L^2(\mu)$, one has $A_0=B\pmod \mu$. 
Thus, $\mathbf{1}_{A_0}\in \{\mathbf{1}_{A} \colon A \in \alpha \in K\}$.
\end{proof}

Finally, we prove the equivalence of precompactness and bounded mean complexity for a collection of partitions.

\begin{prop}\label{prop:precpt-bd-mean-complexity}
Let $(X, \mathcal{B}, \mu)$ be a standard probability space and let $K$ be a nonempty subset of $\mathfrak{P}_r$ with $r\geq 2$.
Then the following assertions are equivalent:
\begin{enumerate}
\item $K$ is precompact in $(\mathfrak{P}_r,\rho)$;
\item $K$ has bounded mean complexity;
\item every sequence in $K$ has bounded mean complexity.
\end{enumerate}
\end{prop} 

\begin{proof}
(1)$\Rightarrow$(2). 
By Lemma~\ref{lem:Rokhlin-metric-eq}, the set $K$ is also precompact in $(\mathfrak{P}_r,\widetilde{\rho})$. 
Fix $\varepsilon>0$.
There exists a finite subset $F$ of $K$  such that 
for any $\alpha\in K$, 
there exists $ \beta_{\varphi(\alpha)} \in F$ such that $\widetilde{\rho}(\alpha,\beta_{\varphi(\alpha ) })< \varepsilon^2 $. 
For each $\beta\in F$, enumerate $\beta =\{B_1^\beta,B_2^\beta,\dotsc,B_r^\beta\}$.
Now we fix a finite subset $E$ of $K$. 
For each $\alpha\in E$, enumerate $\alpha=\{A_1^\alpha,A_2^\alpha,\dotsc,A_r^\alpha\}$ such that 
\[
    \widetilde{\rho}(\alpha,\beta_{\varphi(\alpha)})
    =\sum_{j=1}^r\mu(A_j^\alpha\Delta B_j^{\varphi(\alpha)})<\varepsilon^2.
\]
Since $\alpha$ and $\beta_{\varphi(\alpha)}$ are partitions,
\begin{align*}
2&= \sum_{j=1}^r\mu(A_j^\alpha)+\sum_{j=1}^r \mu(B_j^{\varphi(\alpha)})\\
   & =\sum_{j=1}^r (\mu(A_j^\alpha\cap  B_j^{\varphi(\alpha)}) 
   + \mu(A_j^\alpha\setminus  B_j^{\varphi(\alpha)}) )
   + \sum_{j=1}^r (\mu(B_j^{\varphi(\alpha)} \cap A_j^\alpha)
   + \mu(B_j^{\varphi(\alpha)}\setminus A_j^\alpha) )\\
   &=2 \sum_{j=1}^r \mu(A_j^\alpha\cap  B_j^{\varphi(\alpha)})  + \sum_{j=1}^r\mu(A_j^\alpha\Delta B_j^{\varphi(\alpha)})
   <2 \sum_{j=1}^r \mu(A_j^\alpha\cap  B_j^{\varphi(\alpha)}) +\varepsilon^2.
\end{align*}
This implies 
\[ 
\mu\biggl(\bigcup_{j=1}^r A_j^{\alpha } \cap B_{j}^{\varphi(\alpha ) } \biggr) >1-\varepsilon^2.
 \] 
 
For each $\mathbf{j} = (j_{\beta})_{\beta \in F} \in \{1,\dotsc,r\}^{F}$, 
let
\[ 
B_{\mathbf{j}} = \bigcap_{\beta \in F}  B_{j_\beta}^\beta. 
\]
Note that a point $x\in A_j^{\alpha } \cap B_{j}^{\varphi(\alpha ) }$ 
if and only if  there exists $\mathbf{j}\in\{1,\dotsc,r\}^F$
with $j_{\varphi(\alpha)}=j$
such that $x\in A_j^\alpha\cap B_{\mathbf{j}}$.
Then for each $\alpha\in E$,
\[
\bigcup_{j=1}^r 
A_j^{\alpha } \cap B_{j}^{\varphi(\alpha ) } 
= \bigcup_{\mathbf{j} \in \{1,\dotsc,r\}^{F}}   A^{\alpha}_{j_{\varphi(\alpha)}}\cap B_{\mathbf{j}}.
\]
For each $\mathbf{j} \in \{1,\dotsc,r\}^{F}$, 
let
\[ 
Y_{\mathbf{j}} = \biggl\{ x \in X\colon  
\sum_{\alpha \in E} \mathbf{1}_{A^{\alpha}_{j_{\varphi(\alpha)}}\cap B_{\mathbf{j}} } (x)
\geq (1 - \varepsilon)|E|  \biggr\}
\]
and 
\[
    Y=\bigcup_{\mathbf{j}\in\{1,\dotsc,r\}^F} Y_{\mathbf{j}}.
\]
As $\{B_{\mathbf{j}}\colon \mathbf{j} \in \{1,\dotsc,r\}^{F}\} =\bigvee_{\beta\in F}\beta$ is a finite partition of $X$, 
\begin{align*}
Y&=\biggl\{ x \in X\colon \sum_{\alpha \in E} \mathbf{1}_{\bigcup_{\mathbf{j} \in \{1,\dotsc,r\}^{F}}  A^{\alpha}_{j_{\varphi(\alpha)}}\cap B_{\mathbf{j}} }(x) 
\geq (1 - \varepsilon)|E|  \biggr\}\\
&=
\biggl\{ x \in X\colon \sum_{\alpha \in E} \mathbf{1}_{\bigcup_{j=1}^r  A^{\alpha}_{j}\cap B_{j}^{\varphi(\alpha)} }(x) 
\geq (1 - \varepsilon)|E|  \biggr\}.
\end{align*}
Then
\begin{align*}
(1 - \varepsilon^{2}) |E| 
     &\leq \int_{X} \sum_{\alpha \in E} \mathbf{1}_{\bigcup_{j=1}^r  A^{\alpha}_{j}\cap B_{j}^{\varphi(\alpha)} } \dd\mu \\
     & = \int_{Y} \sum_{\alpha \in E} \mathbf{1}_{\bigcup_{j=1}^r  A^{\alpha}_{j}\cap B_{j}^{\varphi(\alpha)} } \dd\mu  
     + \int_{X\setminus Y} \sum_{\alpha \in E} \mathbf{1}_{\bigcup_{j=1}^r  A^{\alpha}_{j}\cap B_{j}^{\varphi(\alpha)} } \dd\mu \\ 
     & \leq |E|\mu(Y) + (1-\varepsilon)|E|(1-\mu(Y)).
\end{align*}
This implies $\mu(Y)\geq 1 - \varepsilon$. 
For each $x\in X$ and $\mathbf{j}\in \{1,\dotsc,r\}^F$, let 
\[
 E_{\mathbf{j}}^x=\Bigl\{ \alpha \in E \colon x \in  A^{\alpha}_{j_{\varphi(\alpha)}}\cap B_{\mathbf{j}}  \Bigr\}. 
\]
For any $x \in Y_{\mathbf{j}}$, we have $|E_{\mathbf{j}}^x|\geq(1-\varepsilon)|E|$. 
Then for any $\mathbf{j} \in \{1,\dotsc,r\}^{F}$ and $x,y \in Y_{\mathbf{j}} $, 
one has 
\begin{align*}
     H_E(x,y)&=\frac{1}{|E|} |\{\alpha\in E\colon E_x(\alpha)\neq E_y(\alpha)\}| \\
     & \leq \frac{1}{|E|} (|E| - |E_{\mathbf{j}}^x \cap E_{\mathbf{j}}^y|) <3\varepsilon.
\end{align*}
This implies that 
\[\mathcal{C}(E,3\varepsilon) \leq r^{|F|}.\] 
Since the constant $C:=r^{|F|}$ does not depend on $E$,
$K$ has bounded mean complexity.

(2)$\Rightarrow$(3). It is clear.
         
(3)$\Rightarrow$(1). 
Assume that $\{\mathbf{1}_{A} \colon A \in \alpha \in K\}$ is not precompact in $L^2(\mu)$.
Then there exist $\varepsilon>0$ and a sequence  
$(A_i)_{i=1}^{\infty}$ in $\bigcup_{\alpha\in K}\alpha$ such that for any $1\leq i<j$, one has 
$\|\mathbf{1}_{A_i}-\mathbf{1}_{A_j}\|_{L^2(\mu)}\geq \sqrt{20\varepsilon}$. 
Then $\| \mathbf{1}_{A_i}- \mathbf{1}_{A_j}\|_{L^1(\mu)} \geq 20\varepsilon$ for all $1\leq i<j$. 
For each $i\in\bbn$, there exists $\alpha_i\in K$
such that $A_i\in\alpha_i$.
As each partition in $K$ has at most $r$ atoms,
without loss of generality, we assume that $\alpha_i\neq \alpha_j$ for all $1\leq i<j$.
Let $\mathcal{A}=(\alpha_{i})_{i=1}^\infty$.
Then $\mathcal{A}$ is a sequence in $K$.

Since $\mathcal{A}$ has bounded mean complexity,
for every $\varepsilon>0$ there exists $C\in \mathbb{N}$ such that $\mathcal{C}(\{\alpha_1,\dotsc,\alpha_n\},\varepsilon) \leq C$ for all $n\geq 1$.
Fix $n\geq 4\cdot 2^C$ and let $E=\{\alpha_1,\dotsc,\alpha_n\}$.
By the definition of $\mathcal{C}(E,\varepsilon)$, 
there exist points $x_1,x_2,\dotsc,x_C \in X$ such that
\[
\mu\biggl(\bigcup_{j=1}^C B_{H_{E}}(x_j,\varepsilon)\biggr)>1-\varepsilon.
\]
Put $B_1=B_{H_{E}}(x_1,\varepsilon)$ and $B_i=B_{H_{E}}(x_i,\varepsilon) \setminus \bigcup_{j=1}^{i-1}B_{H_{E}}(x_j,\varepsilon)$ for $i\in\{2,\dots,C\}$.
Without loss of generality, we assume $B_i \neq \emptyset$, 
then pick some $y_i \in B_i$ for each $i \in \{1,\dotsc,C\}$. 
Let $B_0=X \setminus \bigcup_{i=1}^C B_i$. 
Then $\{B_0,B_1,\dotsc, B_C\}$ is a finite partition of $X$. 
Note that for any $x,y\in X$,
we have 
\[
\sum_{k=1}^n |\mathbf{1}_{A_{k}}(x)-\mathbf{1}_{A_{k}}(y)| \leq 2 n H_{E}(x, y).
 \]
In particular, for any $x \in B_i$ with $1\leq i \leq C$, 
we have $x,y_i\in B_{H_E}(x_i,\varepsilon)$ and
\[
\sum_{k=1}^n |\mathbf{1}_{A_{k}}(x)-\mathbf{1}_{A_{k}}(y_i)| \leq 2n H_{E}(x, y_i) < 4n \varepsilon.
\]

For $1\leq k \leq  n$,  define
\[
f_{k}(x)=
\begin{cases}
0,& \text{ if }x\in B_0, \\
 \mathbf{1}_{A_{k}}(y_i),& \text{ if } x\in B_i \text{ for some }i\in \{1,2,\dotsc,C\}.
\end{cases}
\] 
Note that $\mu(B_0) <\varepsilon$. Then 
\begin{align*}
    \sum_{k=1}^n  \|f_{k} - \mathbf{1}_{A_{k}}\|_{L^1(\mu)} 
    & =  \sum_{k=1}^n \int_{X} |f_{k}(x)- \mathbf{1}_{A_{k}}(x)| \dd\mu \\
    & = \sum_{k=1}^n \int_{\bigcup_{i=1}^C B_i} 
    |f_{k}(x)- \mathbf{1}_{A_{k}}(x)| \dd\mu +
    \sum_{k=1}^n  \int_{B_0} |f_{k}(x)- \mathbf{1}_{A_{k}}(x)| \dd\mu\\ 
 &\leq \sum_{k=1}^n \sum_{i=1}^C \int_{B_i}|f_k(x)-\mathbf{1}_{A_{k}}(x)| \dd\mu +n\mu(B_0) \\
  &\leq \sum_{i=1}^C \int_{B_i}\sum_{k=1}^n|\mathbf{1}_{A_{k}}(y_i)-\mathbf{1}_{A_{k}}(x)| \dd\mu +n\varepsilon \\
    & \leq \sum_{i=1}^C \int_{B_i}2nH_E(x,y_i) \dd\mu  + n\varepsilon< 4n\varepsilon \sum_{i=1}^C \mu(B_i) + n\varepsilon \leq 5n \varepsilon.
\end{align*}
So there exists a finite set $J \subset \{1, \dotsc, n\}$ with $|J|\geq \frac{n}{2}$ such that 
$\|f_{k} - \mathbf{1}_{A_{k}}\|_{L^1(\mu)}< 10\varepsilon$ for all $k \in J$. 
Since each $f_k$ has the form $ \sum_{i=1}^{C} b_i \mathbf{1}_{B_i}$ with $b_i \in \{0, 1\}$ and $i=1,\dotsc,C$,
$f_k$ has at most $2^C$ choices.
As $|J|\geq \frac{n}{2}>2^C$,
by the pigeonhole principle, there exist $k'\neq k''\in J$ such that $f_{k'}=f_{k''}$. 
Then 
 \[
 \|\mathbf{1}_{A_{k'}}-\mathbf{1}_{A_{k''}}\|_{L^1(\mu)} 
 \leq \|\mathbf{1}_{A_{k'}} -f_{k'} \|_{L^1(\mu)}+ 
 \|f_{k''}-\mathbf{1}_{A_{k''}}\|_{L^1(\mu)} < 20\varepsilon,
 \]
 which is a contradiction.
Thus, the set $\{\mathbf{1}_{A} \colon A \in \alpha \in K\}$ is precompact  in $L^2(\mu)$.
Now by Proposition~\ref{prop:precpt-L2},  $K$ is precompact in $(\mathfrak{P}_r,\rho)$.
\end{proof}

Combining Propositions~\ref{prop:precpt-entropy}, \ref{prop:precpt-L2} and
\ref{prop:precpt-bd-mean-complexity}, we obtain Theorem~\ref{thm:main1}.
We also illustrate Theorem~\ref{thm:main1} with the following two examples.

\begin{exam}
Let $X=\mathbb{R}/\mathbb{Z}$ be equipped with its normalized Haar measure $\mu$. 
For each $t\in[0,\frac{1}{4}]$, define $I_t=[t,t+\tfrac{1}{2})$ and $\alpha_t=\{I_t,I_t^c\}$.
Let $K=\{\alpha_t:0\leq t\leq \frac{1}{4}\}\subseteq\mathfrak{P}_2$.
It is clear that for each $t\in[0,\frac{1}{4}]$, $H_\mu(\alpha_t)=\log 2$. 
However, we show that the collection $K$ has low complexity.

For any $s,t\in[0,\frac14]$, we have
\[
\|\mathbf{1}_{I_s}-\mathbf{1}_{I_t}\|_{L^2(\mu)}^2
=\mu(I_s\triangle I_t)=2|s-t|,
\]
and the same equality holds with $I_s,I_t$ replaced by $I_s^c,I_t^c$. 
Hence, the maps $t\mapsto\mathbf{1}_{I_t}$ and $t\mapsto\mathbf{1}_{I_t^c}$ from $[0,\frac14]$ to $L^2(\mu)$ are continuous. 
This implies that 
\[
\{\mathbf{1}_A:A\in\alpha\in K\}
\]
is precompact in $L^2(\mu)$. 
Thus condition {\rm (3)} of  Theorem~\ref{thm:main1} holds. 
Although not immediately apparent, by  Theorem~\ref{thm:main1} we know that $K$ is precompact in the Rokhlin metric, has zero maximal pattern entropy, and has bounded mean complexity with respect to the Hamming distance.
\end{exam}

\begin{exam}
Let $(p_n)_{n=1}^\infty$ be a sequence in $(0,1)$ with $p_n\to 0$ as $n\to\infty$. 
For each $n\in\mathbb{N}$, let $\mu_n$ be the probability measure on $\{0,1\}$ defined by
\[
\mu_n(\{1\})=p_n\quad \text{and}\quad \mu_n(\{0\})=1-p_n.
\]
Let $X=\{0,1\}^{\mathbb{N}}$, equipped with the product probability measure
\[
\mu=\bigotimes_{n\in\mathbb{N}}\mu_n.
\]
For each $n\geq1$, let $A_n=\{x=(x_k)_{k=1}^\infty\in X:x_n=1\}$ and  define $\alpha_n=\{A_n,A_n^c\}$.
Let $K=\{\alpha_n:n\geq1\}\subseteq\mathfrak{P}_2$.

We show that the collection $K$ has low complexity.
Since
\[
\|\mathbf{1}_{A_n}-\mathbf{1}_{\emptyset}\|_{L^2(\mu)}^2=p_n\to0
\quad\text{and}\quad
\|\mathbf{1}_{A_n^c}-\mathbf{1}_X\|_{L^2(\mu)}^2=p_n\to0,
\]
the characteristic functions of the atoms appearing in $K$ form the union of two convergent sequences in $L^2(\mu)$. Hence $\{\mathbf{1}_A:A\in\alpha\in K\}$ is precompact in $L^2(\mu)$. 
By Theorem~\ref{thm:main1}, $K$ is precompact in the Rokhlin metric, has zero maximal pattern entropy, and has bounded mean complexity with respect to the Hamming distance.
\end{exam}

\section{The complexity of finite partitions in measure-preserving systems}

In this section, we first apply Theorem~\ref{thm:main1} for the complexity of finite partition along a subset of  nonsingular actions and measure-preserving actions on standard probability space.
Then we study the complexity of a partition with respect to a F{\o}lner sequence in the amenable acting group, and prove Theorem~\ref{thm:main2}.

\subsection{The complexity of nonsingular actions}

Let $(X,\mathcal{B},\mu)$ be a standard probability space and $G$ be a discrete group.
If $\Pi\colon G\times X\to X$ is a measurable map satisfying 
\begin{enumerate}
    \item $\Pi(e,x)=x$, where $e$ is the identity of $G$;
    \item $\Pi(gh,x)=\Pi(g,\Pi(h,x))$ for all $g,h\in G$,
\end{enumerate}
then we say that $G$ is \emph{measurably acting} on $X$.
When there is no risk of ambiguity, we write $gx$ for $\Pi(g,x)$.
If for every $g\in G$ and $A\in\mathcal{B}$, 
$\mu(A)=0$ implies $\mu(g^{-1}A)=0$,
then we say that $G$ is a \emph{nonsingular action}.

Assume that a discrete group $G$ is nonsingularly acting on a standard probability space $(X,\mathcal{B},\mu)$.
For a finite partition $\alpha$ of $X$ and $g\in G$, denote $g^{-1}\alpha=\{g^{-1}A \colon A\in \alpha\}$.
Since the action is nonsingular, the pullback of a partition is well defined modulo $\mu$-null sets. 
Hence, if $\alpha\in\mathfrak{P}_r$ for some $r\geq 2$, then $g^{-1}\alpha \in \mathfrak{P}_r$.
For a subset $S$ of $G$, we define the \emph{maximal pattern entropy of $\alpha$ along $S$} by $h_{\mu,S}^*(\alpha)=h_\mu^*(\{g^{-1}\alpha\colon g\in S\})$.
We say that $f\in L^2(\mu)$ is \emph{$S$-almost periodic} 
if $f\circ g\in L^2(\mu)$ for every $g\in S$ and $\{f\circ g\colon g\in S\}$ is precompact in $L^2(\mu)$.

For any nonempty finite subset $F$ of $S$ and $\varepsilon>0$, define 
\[
\mathcal{C}_{\alpha}(F,\varepsilon)=\min\biggl\{|D|\colon 
D\subset X\text{ s.t. }\mu\biggl(\bigcup_{x\in D} B_{H_{F,\alpha}}(x,\varepsilon)\biggr)>1-\varepsilon\biggr\},
\]
where \[
H_{F,\alpha}(x,y)=\frac{1}{|F|} \sum_{g\in F}H_\alpha(gx,gy).
\]
We say that a finite partition $\alpha$ has \emph{bounded mean complexity along $S$} if for any $\varepsilon>0$, there exists a constant $C=C(\varepsilon)\in \mathbb{N}$ such that  $\mathcal{C}_{\alpha}(F,\varepsilon)\leq C$ for any nonempty finite subset $F$ of $S$.

According to Theorem~\ref{thm:main1}, assertions (1)--(4) of the result below follow immediately.
Note that in the measure-preserving setting, the equivalence of (2) and (3) in the following result was proved in \cite[Lemma 3.3]{HSY05} for $\mathbb{Z}$-actions, and
the implication (3)$\Rightarrow$(2) was proved in \cite[Lemma B.2]{LWX25} for $G$-measure-preserving systems.

\begin{prop}\label{prop:measura-actions-1}
Assume that a discrete group $G$ is nonsingularly acting on a standard probability space $(X,\mathcal{B},\mu)$.
For a subset $S$ of $G$ and a finite partition $\alpha$ of $X$,
the following assertions are equivalent:
\begin{enumerate}
    \item $\{g^{-1}\alpha\colon g\in S\}$ is precompact in $(\mathfrak{P},\rho)$;
    \item the maximal pattern entropy of $\alpha$ along $S$ is zero;
    \item for every $A\in\alpha$, $\mathbf{1}_A$ is $S$-almost periodic;
    \item $\{g^{-1}\alpha\colon g\in S\}$ has bounded mean complexity;
    \item $\alpha$ has bounded mean complexity along $S$. 
\end{enumerate}
\end{prop}
\begin{proof}

By Theorem~\ref{thm:main1}, assertions (1)--(4) are equivalent. 
Assume that $\alpha\in\mathfrak{P}_r$ for some $r\geq 2$. 
Let $K=\{g^{-1}\alpha\colon g\in S\}$. 
Then $K\subset \mathfrak{P}_r$.

(5)$\Rightarrow$(4). 
Fix $\varepsilon>0$. 
Let $E$ be a nonempty finite subset of $K$. 
For every $\beta\in E$, choose $g_\beta\in S$ such that $g_\beta^{-1}\alpha=\beta$.
Let $F=\{g_\beta\colon\beta\in E\}$. 
Since the elements of $E$ are distinct, $|F|=|E|$. For every $x,y\in X$, we have $H_{g_\beta^{-1}\alpha}(x,y)=H_\alpha(g_\beta x,g_\beta y)$, and then $H_{E}(x,y) =H_{F,\alpha}(x,y)$.
Therefore, $\mathcal C(E,\varepsilon)=\mathcal{C}_{\alpha}(F,\varepsilon)$. 
By (5), the latter is bounded by a constant depending only on $\varepsilon$. 
Thus $K$ has bounded mean complexity, proving (4).

(1)$\Rightarrow$(5).  
By Lemma~\ref{lem:Rokhlin-metric-eq}, the set $K$ is also precompact in $(\mathfrak{P}_r,\widetilde{\rho})$. 
Fix $\varepsilon>0$.
There exists a finite subset $\{\beta_1,\beta_2,\dotsc,\beta_m\}$ of $K$ such that 
for any $g\in S$, 
there exists $ i(g)\in \{1,2,\dotsc,m\}$ such that $\widetilde{\rho}(g^{-1}\alpha,\beta_{i(g)})< \varepsilon^2 $. 
For each $i \in \{1,2,\dotsc,m\}$, we write $\beta_i =\{B_1^{i},B_2^{i},\dotsc,B_r^{i}\}$.
Now we fix a nonempty finite subset $F$ of $S$. 
For each $g \in F$, enumerate $g^{-1}\alpha=\{A_1^g,A_2^g,\dotsc,A_r^g\}$ such that 
\[
    \widetilde{\rho}(g^{-1}\alpha,\beta_{i(g)})
    =\sum_{j=1}^r\mu(A_j^g\Delta B_j^{i(g)})<\varepsilon^2.
\]
For each $g\in F$, let
\[
M_g=\bigcup_{j=1}^r(A_j^g\cap B_j^{i(g)}).
\]
Then $\mu(M_g)>1-\varepsilon^2$.
Let $\gamma=\bigvee_{i=1}^m\beta_i$, which has at most $r^m$ atoms.
For every atom $B$ of $\gamma$, define
\[
    Y_B=
    \biggl\{x\in X: \sum_{g \in F}\mathbf{1}_{M_g \cap B}(x)>(1-\varepsilon)|F|
    \biggr\}
\]
and $Y=\bigcup_{B\in\gamma}Y_B$.
Since
\begin{align*}
  (1-\varepsilon^2)|F| 
  &\leq  \sum_{g\in F}\mu(M_g)=\int_X\sum_{g\in F}\mathbf{1}_{M_g}(x)\dd\mu \\
    &=\int_Y\sum_{g\in F}\mathbf{1}_{M_g}(x)\dd\mu+\int_{X\setminus Y}\sum_{g\in F}\mathbf{1}_{M_g}(x)\dd\mu\\
    &\leq |F|\mu(Y)+(1-\varepsilon)|F|(1-\mu(Y)).
\end{align*}
Then $\mu(Y)\geq 1-\varepsilon$.
For each $x\in X$ and atom $B$ of $\gamma$, let
\[
F_B^x=\{g\in F:x\in M_g\cap B\}.
\]
For any $x\in Y_B$, we have $|F_B^x|>(1-\varepsilon)|F|$. Then for any $x,y\in Y_B$,   
\[
    H_{F,\alpha}(x,y)
    =\frac1{|F|}\sum_{g\in F}H_\alpha(gx,gy)
    \leq \frac1{|F|}\bigl(|F|-|F_B^x\cap F_B^y|\bigr)<3\varepsilon.
\]
This implies that $\mathcal{C}_{\alpha}(F,3\varepsilon)\leq r^m$.
Since the constant $r^m$ does not depend on $F$, 
$\alpha$ has bounded mean complexity along $S$.
\end{proof}

The maximal pattern entropy of $(X,\mathcal{B},\mu,G)$ along $S$ is defined by $h^*_{\mu,S}(X,G)=\sup_{\alpha\in \mathfrak{P}} h_{\mu,S}^*(\alpha)$.
By Proposition~\ref{prop:measura-actions-1} and expanding the definitions, we have the following consequence.

\begin{prop}\label{prop:measura-actions-2}
Assume that a discrete group $G$ is nonsingularly acting on a standard probability space $(X,\mathcal{B},\mu)$.
For a subset $S$ of $G$,
the following assertions are equivalent:
\begin{enumerate}
    \item the maximal pattern entropy of $(X,\mathcal{B},\mu,G)$ along $S$ is zero;
    \item for every finite partition $\alpha$ of $X$, $\{g^{-1}\alpha\colon g\in S\}$ is precompact in $(\mathfrak{P},\rho)$;
    \item for every finite partition $\alpha$ of $X$ with two atoms, $\{g^{-1}\alpha\colon g\in S\}$ is precompact in $(\mathfrak{P},\rho)$;
    \item for every $A\in\mathcal{B}$, $\mathbf{1}_A$ is $S$-almost periodic;
    \item every finite partition of $X$ has bounded mean complexity along $S$;
 \item every finite partition of $X$ with two atoms has bounded mean complexity along $S$.
\end{enumerate}
\end{prop}

\subsection{The complexity of measure-preserving systems}
Assume that a discrete group $G$ is measurably acting on a standard probability space $(X,\mathcal{B},\mu)$.
If $G$ is also measure-preserving, that is, for every $g\in G$ and $A\in\mathcal{B}$, $\mu(A)=\mu(g^{-1}A)$,
then we say that $(X,\mathcal{B},\mu,G)$ is a \emph{measure-preserving system}. 

\begin{lem}\label{lem:C-E-eps}
Let $(X,\mathcal{B},\mu,G)$ be a measure-preserving system and $\alpha \in \mathfrak{P}_r$ for some $r\geq 2$.
Then for any $h\in G$, nonempty finite subset $F$ of $G$, and $\varepsilon>0$,
$\mathcal{C}_{\alpha}(F,\varepsilon)=\mathcal{C}_{\alpha}(Fh,\varepsilon)$.
\end{lem}
\begin{proof}
Notice that 
\[    
\frac{1}{|Fh|} \sum_{g\in Fh}H_\alpha(gx,gy)=\frac{1}{|F|} \sum_{g\in F}H_\alpha(ghx,ghy),
\]
then $H_{F,\alpha}(x,y)=H_{Fh,\alpha}(h^{-1}x,h^{-1}y)$.
For $x\in X$ and $\varepsilon>0$,
$B_{H_{Fh,\alpha}}(h^{-1}x,\varepsilon)=h^{-1} B_{H_{F,\alpha}}(x,\varepsilon)$.

Pick a finite subset $D$ of $X$ with $|D|=\mathcal{C}_{\alpha}(F,\varepsilon)$.
Then 
\[
    \mu\biggl(\bigcup_{x\in h^{-1}D} B_{H_{Fh,\alpha}}(x,\varepsilon)\biggr) = \mu\biggl(h^{-1} \bigcup_{x\in D} B_{H_{F,\alpha}}(x,\varepsilon)\biggr)
    = \mu\biggl(\bigcup_{x\in D} B_{H_{F,\alpha}}(x,\varepsilon)\biggr)>1-\varepsilon.
\]
This implies that $\mathcal{C}_{\alpha}(Fh,\varepsilon)\leq \mathcal{C}_{\alpha}(F,\varepsilon)$. 
By the same reason, one also has  $\mathcal{C}_{\alpha}(F,\varepsilon) \leq \mathcal{C}_{\alpha}(Fh,\varepsilon)$.
Thus, $\mathcal{C}_{\alpha}(F,\varepsilon) = \mathcal{C}_{\alpha}(Fh ,\varepsilon)$.
\end{proof}

\begin{lem}\label{lem:partition-entropy}
Let $(X,\mathcal{B},\mu,G)$ be a measure-preserving system and $S\subset G$. 
Then the map $\mathfrak{P}\to [0,\infty)$, $\alpha \mapsto h^*_{\mu,S}(\alpha)$, is continuous with respect to $\rho$.
In particular, $\{\alpha\in\mathfrak{P}\colon h^*_{\mu,S}(\alpha)=0\}$ is closed in $(\mathfrak{P},\rho)$.
\end{lem}
\begin{proof}
For any $\alpha,\beta \in \mathfrak{P}$, we claim that
\[
    h^*_{\mu,S}(\alpha) \leq h^*_{\mu,S}(\beta) + H_{\mu}(\alpha |\beta).
\]
In fact, for any finite subset $S_1$ of $S$, 
\begin{align*}
    H_{\mu} \biggl(\bigvee_{g\in S_1} g^{-1}\alpha\biggr) 
    &\leq H_{\mu}\biggl(\bigvee_{g\in S_1} g^{-1}\alpha \vee \bigvee_{h\in S_1} h^{-1}\beta\biggr )  \\
    &= H_{\mu}\biggl(\bigvee_{h\in S_1} h^{-1}\beta \biggr )+H_{\mu} \biggl(\bigvee_{g\in S_1} g^{-1}\alpha  |\bigvee_{h\in S_1} h^{-1}\beta\biggr ) \\  
    &\leq H_{\mu}\biggl(\bigvee_{h\in S_1} h^{-1}\beta\biggr )
    +\sum_{g\in S_1} H_{\mu}\biggl(g^{-1}\alpha | \bigvee_{h\in S_1} h^{-1}\beta\biggr) \\
    &\leq  H_{\mu}\biggl(\bigvee_{h\in S_1} h^{-1}\beta\biggr )
    +\sum_{g\in S_1} H_{\mu}\biggl(g^{-1}\alpha |g^{-1}\beta\biggr)\\
    &=H_{\mu}\biggl(\bigvee_{h\in S_1} h^{-1}\beta\biggr )
    + |S_1| H_{\mu}(\alpha|\beta).
\end{align*}
This implies that 
$h^*_{\mu,S}(\alpha) \leq h^*_{\mu,S}(\beta) + H_{\mu}(\alpha |\beta)$.

By symmetry, we also have $h^*_{\mu,S}(\beta) \leq h^*_{\mu,S}(\alpha) + H_{\mu}(\beta |\alpha)$.
Thus, 
$|h^*_{\mu,S}(\alpha) - h^*_{\mu,S}(\beta)|\leq \rho(\alpha,\beta) $.
This ends the proof.
\end{proof}

\begin{prop}
Let $(X,\mathcal{B},\mu,G)$ be a measure-preserving system and let
$S\subset G$. Define
\[
\mathcal{K}_S
:=\bigl\{A\in\mathcal{B}:
h^*_{\mu,S}(\{A,A^c\})=0\bigr\}.
\]
Then $\mathcal{K}_S$ is a sub-$\sigma$-algebra of $\mathcal{B}$, and
the space of $S$-almost periodic functions is precisely
\[
L^2(X,\mathcal{K}_S,\mu).
\]
In particular, the space of $S$-almost periodic functions is a closed
linear subspace of $L^2(\mu)$. Moreover, if $S=G$, then
$\mathcal{K}_S$ is $G$-invariant.
\end{prop}

\begin{proof}

Here and below, we write $L^2(\mu)=L^2(X,\mathcal{B},\mu)$.
Let $H$ be the collection of $S$-almost periodic functions in $L^2(\mu)$. 

First, we show that $H$ is a closed linear subspace of  $L^2(\mu)$. 
It is easy to verify that $H$ is a linear subspace.
Now we show that $H$ is closed. 
Consider a sequence $(f_i)_{i=1}^\infty$ in $H$ which converges to $f$ in $L^2(\mu)$. 
For every $\varepsilon>0$, 
there exists some $n\in\mathbb{N}$ such that 
$\|f_n-f\|_{L^2(\mu)}<\frac{\varepsilon}{3}$.
Since $f_n$ is $S$-almost periodic, 
there exists a finite subset $S'$ of $S$ such that 
for any $g\in S$ there exists some $g'\in S'$ with 
$\|f_n\circ g-f_n\circ g'\|_{L^2(\mu)}<\frac{\varepsilon}{3}$.
As $G$ is measure-preserving,  
\begin{align*}
 \|f\circ g-&f\circ g'\|_{L^2(\mu)} \\
 \leq & \|f\circ g- f_n\circ g\|_{L^2(\mu)}+ \|f_n\circ g-f_n\circ g'\|_{L^2(\mu)}+ \|f_n\circ g'- f\circ g'\|_{L^2(\mu)}\\
=& \|f-f_n\|_{L^2(\mu)}+ \|f_n\circ g-f_n\circ g'\|_{L^2(\mu)}+ 
\|f_n-f\|_{L^2(\mu)}\\
<&\frac{\varepsilon}{3}+\frac{\varepsilon}{3}+\frac{\varepsilon}{3}=\varepsilon.
\end{align*} 
Then $f$ is $S$-almost periodic.

Next, we show that $\mathcal{K}_S$ is a sub-$\sigma$-algebra of $\mathcal{B}$ and that, when $S=G$, it is $G$-invariant.
Recall that $\mathcal{K}_S =\bigl\{A\in\mathcal{B}: h^*_{\mu,S}(\{A,A^c\})=0\bigr\}$.
It is clear that 
$X\in \mathcal{K}_S$ and $\emptyset\in \mathcal{K}_S$.
For any $A\in\mathcal{K}_S$, since $(A^c)^c=A$, 
we have $A^c\in \mathcal{K}_S$.
Assume that $A,B\in \mathcal{K}_S$.
Note that for any finite subset $S'$ of $S$, 
\[
   \bigvee_{g\in S'} g^{-1}\{A\cup B, (A\cup B)^c\}  \preceq\bigvee_{g\in S'} g^{-1}\{A,A^c\} \vee \bigvee_{g\in S'}  g^{- 1}\{B,B^c\} 
\]
and then 
\[
    H_\mu\biggl(\bigvee_{g\in S'} g^{-1}\{A\cup B, (A\cup B)^c\}  \biggr )\leq H_\mu\biggl(\bigvee_{g\in S'} g^{-1}\{A,A^c\}\biggr) +
    H_\mu\biggl( \bigvee_{g\in S'}  g^{- 1}\{B,B^c\}\biggr ). 
\]
This implies that 
$h_{\mu,S}^*(\{A\cup B, (A\cup B)^c\})\leq h_{\mu,S}^*(\{A, A^c\})+h_{\mu,S}^*(\{B, B^c\}) $.
Then $h_{\mu,S}^*(\{A\cup B, (A\cup B)^c\})=0$ and $A\cup B\in\mathcal{K}_S$.
This shows that $\mathcal{K}_S$ is a subalgebra of $\mathcal{B}$.

To show that $\mathcal{K}_S$ is a sub-$\sigma$-algebra of $\mathcal{B}$,
we only need to show that $\mathcal{K}_S$ is a monotone class.
Let $(A_i)_{i=1}^\infty$ be an increasing sequence in $\mathcal{K}_S$ and $A=\bigcup_{i=1}^\infty A_i$.
For each $i\in\mathbb{N}$, let $\alpha_i=\{A_i,A_i^c\}$ and let
$\alpha =\{A,A^c\}$.
It is clear that $\alpha_i\to \alpha$ in $(\mathfrak{P}_2,\widetilde{\rho})$ as $i\to\infty$.
By Lemma~\ref{lem:partition-entropy},
$h^*_{\mu,S}(\alpha_i)\to h^*_{\mu,S}(\alpha)$ as $i\to\infty$.
Since each $A_i\in \mathcal{K}_S$, $h^*_{\mu,S}(\alpha_i)=0$.
Then $h^*_{\mu,S}(\alpha)=0$ and $A\in\mathcal{K}_S$.
Therefore $\mathcal{K}_S$ is a sub-$\sigma$-algebra of
$\mathcal{B}$.

As $G$ is measure-preserving, it is easy to see that for any $g\in G$, $h^*_{\mu,G}(\{A,A^c\})=h^*_{\mu,G}(g^{-1}\{A,A^c\})$.
So if $S=G$, 
then the sub-$\sigma$-algebra $\mathcal{K}_S$ is $G$-invariant.
 
Finally, we show that $H=L^2(X,\mathcal{K}_S,\mu)$. 
By Proposition~\ref{prop:measura-actions-1} $A\in\mathcal{K}_S$ if and only if $\mathbf{1}_A\in H$.
It follows that every $\mathcal{K}_S$-measurable simple function belongs to $H$.  
Such simple functions are dense in $L^2(X,\mathcal{K}_S,\mu)$, and $H$ is closed in $L^2(\mu)$.
Therefore $L^2(X,\mathcal{K}_S,\mu)\subset H$.

For the reverse inclusion, first let $f\in H$ be real-valued.
For $a\in\mathbb{R}$ and $n\in\mathbb{N}$ with $n\geq1$, 
set
\[
    \phi_{a,n}(t)=
    \begin{cases}
        0, & t\leq a,\\
        n(t-a), & a<t<a+\dfrac{1}{n},\\
        1, & t\geq a+\dfrac{1}{n}.
    \end{cases}
\]
The function $\phi_{a,n}$ is $n$-Lipschitz. Hence, for any real-valued functions
$u,v\in L^2(\mu)$ and $\mu$-a.e.\ $x\in X$,
\[
    |\phi_{a,n}(u(x))-\phi_{a,n}(v(x))|
    \leq n|u(x)-v(x)|.
\]
Consequently,
\begin{align*}
    \|\phi_{a,n}\circ u-\phi_{a,n}\circ v\|_{L^2(\mu)} \leq n\|u-v\|_{L^2(\mu)}.
\end{align*}
To verify that $\phi_{a,n}\circ f\in H$, fix $\varepsilon>0$ and choose a finite $\varepsilon/n$-net $\{u_j\colon1\leq j\leq m\} \subset \{f\circ g\colon g\in S\}$.  
For every $g\in S$, there exists some $u_j$ such that
\begin{align*}
    \| (\phi_{a,n}\circ f)\circ g-\phi_{a,n}\circ u_j\|_{L^2(\mu)}
    &=\|\phi_{a,n}\circ(f\circ g)-\phi_{a,n}\circ u_j\|_{L^2(\mu)}\\
    &\leq n\|f\circ g-u_j\|_{L^2(\mu)}<\varepsilon.
\end{align*}
Thus  $\{\phi_{a,n}\circ f \circ g : g \in S \}$ is totally bounded. Then $\phi_{a,n}\circ f\in H$.

As $n\to\infty$, $\phi_{a,n}\circ f$ converges pointwise to
$\mathbf{1}_{\{f>a\}}$.  Since $0\leq\phi_{a,n}\circ f\leq1$, 
the dominated convergence theorem yields $ \phi_{a,n}\circ f \to \mathbf{1}_{\{f>a\}}$ in $L^2(\mu)$.
The closedness of $H$ now implies $\mathbf{1}_{\{f>a\}}\in H$, 
and hence $\{f>a\}\in\mathcal{K}_S$, for every $a\in\mathbb{R}$.
Thus $f$ is $\mathcal{K}_S$-measurable.

Finally, suppose that $f\in H$ is complex-valued.  Complex conjugation is an isometry on $L^2(\mu)$ and maps the orbit of $f$ onto the orbit of $\overline f$.  
Thus $\overline f\in H$. 
Since $H$ is a complex linear subspace,
\[
    \operatorname{Re}f=\frac{f+\overline f}{2}\in H
    \quad\text{and}\quad
    \operatorname{Im}f=\frac{f-\overline f}{2i}\in H.
\]
The real-valued case shows that both $\operatorname{Re}f$ and
$\operatorname{Im}f$ are $\mathcal{K}_S$-measurable, so $f$ is also $\mathcal{K}_S$-measurable.  
Therefore $H\subset L^2(X,\mathcal{K}_S,\mu)$.
Together with $L^2(X,\mathcal{K}_S,\mu)\subset H$ proved above, this gives $ H=L^2(X,\mathcal{K}_S,\mu)$.
\end{proof}

\begin{rem}
If in addition $(X,\mathcal{B},\mu,G)$ is a measure-preserving  system, then the assertions in Proposition~\ref{prop:measura-actions-2} are also equivalent to 
\begin{enumerate}
    \item[(4')] every $f\in L^2(\mu)$ is $S$-almost periodic.
\end{enumerate}
\end{rem} 

\subsection{The complexity of amenable group actions}
We say that a countably infinite discrete group $G$ is \emph{amenable} if there is a sequence of nonempty finite subsets $(F_n)_{n=1}^\infty$ of $G$ such that for every $g \in G$, $\lim_{n\to\infty} \frac{|g F_n \triangle F_n|}{|F_n|} = 0$. 
Such a sequence $(F_n)_{n=1}^\infty$ is called a  \emph{F{\o}lner sequence} for $G$. 
If in addition there exists $c>0$ such that $|\bigcup_{1\leq i<n}F_i^{-1}F_n|\leq c|F_n|$ for all $n\geq 2$, then we say that the F{\o}lner sequence $(F_n)_{n=1}^\infty$ is \emph{tempered}.
Note that every F{\o}lner sequence has a tempered subsequence, see \cite[Proposition 1.4]{L01}.

We say that a function $f\in L^2(\mu)$ is \emph{almost periodic} if $\{f\circ g\colon g\in G\}$ is precompact in $L^2(\mu)$.
For a finite partition $\alpha$ of $X$, 
\emph{the maximal pattern entropy} of $\alpha$ is defined by 
$h^*_\mu(\alpha)=h^*_{\mu,G}(\alpha)$.
We say that a finite partition $\alpha$ of $X$ has \emph{bounded mean complexity with respect to $(F_n)_{n=1}^\infty$}  
if for any $\varepsilon>0$, there exists a constant $C=C(\varepsilon)\in \mathbb{N}$ such that  
$\mathcal{C}_{\alpha}(F_n,\varepsilon) \leq C$ for any $n\in \bbn$, 
and $\alpha$ is \emph{mean equicontinuous with respect to $(F_n)_{n=1}^\infty$}
if for any $\varepsilon>0$, there exist 
pairwise disjoint measurable subsets $B_1,B_2,\dotsc,B_k$
with $\mu(\bigcup_{i=1}^k B_i)>1-\varepsilon$ such that 
for every $x,y\in B_i$ for some $i\in\{1,2,\dotsc,k\}$, one has 
\[  
\limsup_{n \to \infty}   H_{F_n,\alpha} (x,y) <\varepsilon.
\] 

\begin{prop}\label{prop:mps-mean-cp}
Let $G$ be a countably infinite discrete amenable group and $(F_n)_{n=1}^\infty$ be a F{\o}lner sequence in $G$. 
For a measure-preserving system $(X,\mathcal{B},\mu, G)$,
a finite partition $\alpha$ of $X$ has bounded mean complexity along $G$ if and only if $\alpha$ has bounded mean complexity with respect to $(F_n)_{n=1}^\infty$.
\end{prop}
\begin{proof}
($\Rightarrow$) 
It follows from the definition directly.

($\Leftarrow$) 
Fix  $0<\varepsilon<\frac{1}{2}$.
By the assumption that $\alpha$ has bounded mean complexity with respect to $(F_n)_{n\in \bbn}$, 
there is $C=C(\varepsilon) \in \bbn$ such that for any $n\in\bbn$,
there is a finite set $D_n\subset X$ with $|D_n|\leq C$ 
such that
\[
    \mu\biggl(\bigcup_{x\in D_n} B_{H_{F_n,\alpha}}(x,\varepsilon^2)\biggr)>1-\varepsilon^2.
\]
Now fix a nonempty finite subset $F$ of $G$.
Since $\{F_n\}_{n\in \mathbb{N}}$ is a F{\o}lner sequence,  
there exists $m\in \bbn$ such that for any $g\in F$, 
\[
    \frac{|gF_m \Delta F_m|}{|F_m|}<\frac{1}{2|F|}.
\]
For each $g\in F$, 
let $H_g=\{h \in F_m : gh \in F_m\}$.
Then $|H_g|\geq(1-\frac{1}{2|F|})|F_m|$. 
So $H=\bigcap_{g\in F} H_g$ satisfies 
$|H|\geq\frac{1}{2}|F_m|$.
It is clear that $h\in H$ if and only if $Fh\subset F_m$.

Enumerate $D_m=\{x_1,\dotsc,x_\ell\}\subset X$.
Let 
\[
B_1=B_{H_{F_m,\alpha}}(x_1,\varepsilon^2)
\text{ and }
    B_k=B_{H_{F_m,\alpha}}(x_k,\varepsilon^2) \setminus \bigcup_{j=1}^{k-1} B_{H_{F_m,\alpha}}(x_j,\varepsilon^2) \text{ for } k=2,\dotsc,\ell.
\]
Then
\[
\mu\biggl(\bigcup_{i=1}^\ell B_i\biggr)=
\mu\biggl(\bigcup_{x\in D_m} B_{H_{F_m,\alpha}}(x,\varepsilon^2)\biggr)>1-\varepsilon^2.
\]
For any $1\leq k \leq \ell$ and $ h \in H$, 
let
\[
    B_k^h=\{x \in B_k : H_{Fh,\alpha}(x, x_k)<\varepsilon\}.
\]
For any $x \in B_k$, let $H_k(x)=\{h\in H \colon x \in B_k^h\}$.
Thus,
\begin{align*}
    \varepsilon \frac{|H\setminus H_k(x)|}{|H|} &\leq \frac{1}{|H|}\sum_{h \in H} H_{Fh,\alpha}(x, x_k) = \frac{1}{|H|}\sum_{h \in H} \frac{1}{|Fh|}\sum_{g \in Fh} H_{\alpha}(gx, gx_k) \\
    &\leq \frac{1}{|F||H|}\sum_{g \in F}\sum_{h \in H} H_{\alpha}(ghx, ghx_k) 
    \leq
    \frac{1}{|F||H|}\sum_{g\in F}\sum_{t \in F_m }H_{\alpha}(tx,tx_k) \\
    &\leq \frac{1}{|F||H|}|F|\sum_{t \in F_m} H_{\alpha}(tx,tx_k)\leq \frac{2}{|F_m|}\sum_{t \in F_m} H_{\alpha}(tx, tx_k) \\
    &\leq 2 H_{F_m, \alpha}(x, x_k) < 2\varepsilon^2.
\end{align*}
Then $|H_k(x)|>(1-2\varepsilon)|H|$. So
\begin{align*}
    \sum_{h \in H}\mu(B^h_k)= \int_{B_k} \sum_{h \in H} \mathbf{1}_{B^h_k}(x) d\mu =\int_{B_k} |H_k(x)| d\mu > (1-2\varepsilon)|H| \mu(B_k).
\end{align*}

We show that there is $h \in  H $ such that $\mu(\bigcup_{k=1}^\ell B_k^h)>1-3\varepsilon$.
If not, then for any $h \in H$, $\mu(\bigcup_{k=1}^\ell B_k^h)\leq 1-3\varepsilon$. 
Thus
\begin{align*}
    (1-2\varepsilon)|H| \mu\biggl(\bigcup_{k=1}^\ell B_k\biggr) 
    \leq \sum_{h \in H}\mu\biggl(\bigcup_{k=1}^\ell B_k^h\biggr) 
    \leq |H|(1-3\varepsilon),
\end{align*}
and  
\[\mu\biggl(\bigcup_{k=1}^\ell B_k\biggr) \leq \frac{1-3\varepsilon}{1-2\varepsilon}.
\]
As $\varepsilon\in (0,\frac{1}{2})$,
it is easy to check that $\frac{1-3\varepsilon}{1-2\varepsilon}<1-\varepsilon^2$.
Then the above formula contradicts the fact $\mu\bigl(\bigcup_{i=1}^\ell B_i\bigr)>1-\varepsilon^2$.
So there is some $h \in H$ such that 
$\mu(\bigcup_{k=1}^\ell B_k^h)>1-3\varepsilon$,
which implies that $\mathcal{C}_{\alpha}(Fh,3\varepsilon) \leq C$.
By Lemma~\ref{lem:C-E-eps}, $\mathcal{C}_{\alpha}(F,3\varepsilon)=\mathcal{C}_{\alpha}(Fh,3\varepsilon) \leq C$. 
Therefore, $\alpha$ has bounded mean complexity along $G$. 
\end{proof}

\begin{prop} \label{prop:mps-mean-eq}
Let $G$ be a countably infinite discrete amenable group and $(F_n)_{n=1}^\infty$ be a F{\o}lner sequence in $G$. 
For a measure-preserving system $(X,\mathcal{B},\mu, G)$
and a finite partition $\alpha$ of $X$,
\begin{enumerate}
    \item if $\alpha$ is mean equicontinuous with respect to $(F_n)_{n=1}^\infty$, then it has bounded mean complexity with respect to $(F_n)_{n=1}^\infty$;
    \item if $(F_n)_{n=1}^\infty$ is tempered and $\alpha$ has bounded mean complexity with respect to $(F_n)_{n=1}^\infty$, then $\alpha$ is mean equicontinuous with respect to $(F_n)_{n=1}^\infty$.
\end{enumerate}
\end{prop}
\begin{proof}
(1) 
A similar argument to that in Lemma~\ref{lem:partition-seq-mean-equi} yields the following result.
For any $\varepsilon>0$, there exist 
pairwise disjoint measurable subsets $B_1,B_2,\dotsc,B_k$
with $\mu(\bigcup_{i=1}^k B_i)>1-\varepsilon$ such that 
for every $x,y\in B_i$ for some $i\in\{1,2,\dotsc,k\}$, one has $H_{F_n,\alpha} (x,y) <\varepsilon$ for all $n\in\mathbb{N}$. 
For each $i\in\{1,2,\dotsc,k\}$, pick $x_i\in B_i$.
Then $B_i\subset B_{H_{F_n,\alpha}}(x_i,\varepsilon)$ for all $n\in\mathbb{N}$ and $i\in\{1,2,\dotsc,k\}$.
This shows that $\mathcal{C}_{\alpha}(F_n, \varepsilon)\leq k$ for all $n\in\mathbb{N}$.
Thus, $\alpha$ has bounded mean complexity with respect to $(F_n)_{n=1}^\infty$. 

(2) 
As $\alpha$ has bounded mean complexity with respect to $(F_n)_{n=1}^\infty$,
for any $\varepsilon>0$, there exists a constant $C\in \mathbb{N}$ such that for any $n\in\mathbb{N}$ there exists a finite subset $D_n$ of $X$ with $|D_n|\leq C$ such that 
\[
\mu\biggl(\bigcup_{x\in D_n} B_{H_{F_n,\alpha}}(x,\varepsilon)\biggr) >1-\varepsilon.
\]
Enumerate $\alpha$ as $\{A_1,A_2,\dotsc,A_r\}$.
Let $Y=X\times X \setminus \bigcup_{i=1}^r A_i \times A_i$.
Then for any $x,y\in X$,
\[
    H_{F_n,\alpha}(x,y)=\frac{1}{|F_n|}\sum_{g \in F_n}\mathbf{1}_Y(gx,gy).
\]
Since the F{\o}lner sequence $(F_n)_{n\in \bbn}$ is tempered, 
applying the pointwise ergodic convergence theorem (see \cite[Theorem 1.2]{L01}),
one has that  the limit 
\[
    \lim_{n\to\infty} H_{F_n,\alpha}(x,y)
\]
exists for $\mu\times \mu$-a.e. $(x,y)\in X\times X$.
Thus, for a given $0<\eta < \frac{\varepsilon}{2C}$, 
by Egorov's theorem there exists a subset $R$ of $X\times X$ with $\mu\times \mu(R)>1-\eta^2$ and $N\in\bbn$
such that for any $(x,y)\in R$, one has
\[
|H_{F_n,\alpha}(x,y)-H_{F_N,\alpha}(x,y)|< \eta,\quad \forall\ n\geq N.  
\]
Then by Fubini's theorem there is a measurable subset $X_0$ of $X$ with $\mu(X_0)>1-\eta$ such that for any $x\in X_0$,
$\mu(R_x)>1-\eta$, where $R_x=\{y\in X\colon (x,y)\in R\}$.
Enumerate $D_{N}=\{z_1, \dotsc, z_m\}$. 
Clearly $m\leq C$.
Let
\[
I=\bigl\{1\leq i\leq m\colon 
X_0\cap B_{H_{F_N,\alpha}} (z_i, \varepsilon) \neq\emptyset\bigr\}.
\]
Then $|I|\leq m$. 
For each $i\in I$, pick $y_i\in X_0\cap  B_{H_{F_N,\alpha}}(z_i, \varepsilon)$.
For each $i\in\{1,2,\dotsc,m\}$, let
\[
    \tilde{B}_i= X_0 \cap \bigcap_{j\in I} R_{y_j}
\cap   B_{H_{F_N,\alpha}}(z_i, \varepsilon).
\]
Let $B_1=\tilde{B}_1$ and $B_i=\tilde{B}_i\setminus \bigcup_{j=1}^{i-1} \tilde{B}_j$ for $i=2,\dotsc,m$.
Then 
\begin{align*}
\mu\biggl(\bigcup_{i=1}^m B_i\biggr)&=
\mu\biggl( X_0 \cap \bigcap_{j\in I} R_{y_j}
\cap \bigcup_{i=1}^m B_{H_{F_N,\alpha}}(z_i, \varepsilon)
\biggr)\\
&\geq 1-\eta-C\eta-\varepsilon>1-2\varepsilon.
\end{align*}
For any $x,y\in X$ and $n\geq N$, 
if $x,y\in B_i$ for some $i\in\{1,2,\dotsc,m\}$,
then 
\begin{align*}
    H_{F_n,\alpha}(x,y)&\leq H_{F_n,\alpha}(x,y_i)+ H_{F_n,\alpha}(y_i,y)\\
    &\leq H_{F_N,\alpha}(x,y_i)+\eta+ H_{F_N,\alpha}(y_i,y)+\eta \\
    &\leq 2\varepsilon +\eta+2\varepsilon+\eta <5\varepsilon .
\end{align*} 
This implies that $\alpha$ is mean equicontinuous with respect to $(F_n)_{n=1}^\infty$.
\end{proof}

\begin{rem}
In the proof of Proposition~\ref{prop:mps-mean-eq}(2),
temperedness is used only to apply the pointwise ergodic theorem to the diagonal action on $X\times X$. 
It is therefore a technical assumption required by our argument. 
Our proof does not cover arbitrary non-tempered F{\o}lner sequences, leaving open within the present argument the question of whether the temperedness assumption is necessary for the implication itself.

\end{rem}

Combining Propositions~\ref{prop:measura-actions-1}, \ref{prop:mps-mean-cp} and~\ref{prop:mps-mean-eq}, we finish the proof of Theorem~\ref{thm:main2}.

By \cite[Theorem 3.12]{G03}, $(X,\mathcal{B},\mu, G)$ has  discrete spectrum  if and only if every $f\in L^2(\mu)$ is almost periodic.
The \emph{maximal pattern entropy} of $(X,\mathcal{B},\mu, G)$ is defined by $h^*_\mu(X,G)=\sup_{\alpha\in\mathfrak{P}}h^*_\mu(\alpha)$. 
Applying Theorem~\ref{thm:main2}, we have the following characterizations of discrete spectrum.
Note that the equivalence of (1), (4) and (5) in Theorem~\ref{thm:mps-equi-result} was already proved in \cite[Theorem 1.3]{YZZ21}, 
but the remaining equivalent characterizations were previously unknown.

\begin{thm}\label{thm:mps-equi-result} 
Let $(X,\mathcal{B}_X,\mu, G)$ be a measure-preserving system with $G$ being a countably infinite discrete amenable group.
Then the following assertions are equivalent:
\begin{enumerate}
\item $(X,\mathcal{B}_X,\mu, G)$ has discrete spectrum;
\item the maximal pattern entropy of $(X,\mathcal{B}_X,\mu, G)$ is zero;
\item  every finite partition of $X$ has bounded mean  complexity along $G$;
\item every finite partition of $X$ has bounded mean  complexity with respect to any F{\o}lner sequence of $G$;
\item every finite partition of $X$ with two atoms has bounded mean  complexity with respect to some F{\o}lner sequence of $G$;
\item every finite partition of $X$ is  mean equicontinuous with respect to any tempered F{\o}lner sequence of $G$;
\item every finite partition of $X$  with two atoms is  mean equicontinuous with respect to some tempered F{\o}lner sequence of $G$.
\end{enumerate}
\end{thm}

\begin{rem}
Theorem~\ref{thm:mps-equi-result} gives a global characterization of discrete spectrum, 
where the corresponding low-complexity property is required to hold for all 
finite measurable partitions. 
By contrast, Theorem~\ref{thm:main2} provides a local characterization at the level of finite partition families. Thus, 
our results extend the classical system-level characterization of discrete spectrum to a finer partition-level setting.
\end{rem}

\section{The complexity of invariant measures of a topological dynamical system}

It is well known that every measure-preserving system has a topological model, that is, it is metrically isomorphic to an invariant measure of a topological dynamical system, see e.g. \cite[Theorem 2.18]{G03}.
Hence,  in this section we study the complexity of an invariant measure in a topological dynamical system using the metric.
In what follows, we shall see that the metric is  only capable of defining global complexity, so our local results on partitions essentially extend the main results in \cite{XX24} and \cite{YZZ21}.

Let $(X,d)$ be a compact metric space and $G$ be a discrete group.
If a continuous map $\Psi\colon G\times X\to X$ satisfies 
\begin{enumerate}
    \item $\Psi(e,x)=x$, where $e$ is the identity of $G$;
    \item $\Psi(gh,x)=\Psi(g,\Psi(h,x))$ for all $g,h\in G$,
\end{enumerate}
then we say that $(X,d,G)$ is a \emph{topological dynamical system}. When there is no risk of ambiguity, we write $gx$ for $\Psi(g,x)$.

For a  nonempty finite subset $F$ of $G$, we define a metric $\bar{d}_F$ as follows: for every $x,y\in X$,
\[
    \bar{d}_F(x,y) = \frac{1}{|F|}\sum_{g\in F}d(gx,gy).
\]
For $x\in X$ and $\varepsilon>0$, denote
\[
    B_{\bar{d}_F}(x,\varepsilon)=\{y\in X\colon \bar{d}_F(x,y)<\varepsilon\}.
\]
Let $\mu$ be a Borel probability measure on $X$ and $\mathcal{B}_X$ be the Borel $\sigma$-algebra on $X$.
For a nonempty finite subset $F$ of $G$ and $\varepsilon>0$,
denote
\[
    \mathcal{C}_d(F,\varepsilon)=
    \min\biggl\{|D|\colon D\subset X \text{ s.t. } \mu\biggl(\bigcup_{x\in D} B_{\bar{d}_F}(x,\varepsilon) \biggr)>1-\varepsilon\biggr\}.
\]
For a subset $S$ of $G$, we say that $\mu$ has \emph{bounded $(S,d)$-mean complexity} if 
for every $\varepsilon>0$, there exists a constant $C=C(\varepsilon)\in \mathbb{N}$ such that  $\mathcal{C}_d(F,\varepsilon)\leq C$ for any nonempty finite subset $F$ of $S$.

\begin{rem}
As $X$ is compact, there exists a finite partition $\alpha$ of $X$ such that the diameter of each atom in $\alpha$ is less than $\varepsilon$.
It is easy to check that $ \mathcal{C}_d(F,\varepsilon)\leq |\bigvee_{g\in F}g^{-1}\alpha|\leq |\alpha|^{|F|}$.
Then  $\mu$ has  bounded $(S,d)$-mean complexity if 
for every $\varepsilon>0$, there exists a constant $C\in \mathbb{N}$ and $N\in\mathbb{N}$ such that  $\mathcal{C}_d(F,\varepsilon)\leq C$ for all nonempty finite subsets $F$ of $S$ with $|F|\geq N$.
\end{rem} 

We say that $(X,\mathcal{B}_X,\mu,G)$ has \emph{bounded mean complexity along $S$} if every finite partition $\alpha$ of $X$ has bounded mean complexity along $S$.
The concept of bounded $(S,d)$-mean complexity is the bounded $S$-max-mean-complexity in \cite{XX24}.
The following result is equivalent to the main result in \cite{XX24}.
Here we provide a new proof via the complexity of finite partitions. 

\begin{thm}\label{thm:tds-mps-cp}
Let $(X,d,G)$ be a topological dynamical system and $\mu$ be a $G$-invariant measure on $X$.
For a subset $S$ of $G$,
$\mu$ has bounded $(S,d)$-mean complexity if and only if $(X,\mathcal{B}_X,\mu,G)$ has bounded mean complexity along $S$.
\end{thm}

\begin{proof}
($\Rightarrow $) 
Fix a finite partition $\alpha=\{A_1,\dotsc,A_r\}$ of $X$.
Without loss of generality, we assume that $\mu(A_i)>0$ for any $1\leq i\leq r$. 
If $r=1$, then $H_{F,\alpha}=0$ for every nonempty finite $F\subset S$,
so $\alpha$ has bounded mean complexity along $S$. Hence, assume that $r\geq2$.
Fix an arbitrary $\varepsilon\in (0,1)$.
By the regularity of $\mu$, for any $A_i$,
there exists a compact subset $P_i$ of $A_i$
such that $\mu(A_i\setminus P_i)<\frac{\mu(A_i)\varepsilon^2}{2}$.
Let $U_0=X\setminus \bigcup^{r}_{i=1}P_i$.
Then $\mu(U_0)<\frac{\varepsilon^2}{2}$.
Let $\delta =\min \{ d(P_i,P_j),\frac{\varepsilon}{2}\colon 1\leq i<j\leq r\}$.

Since $\mu$ has bounded $(S,d)$-mean complexity, there exists $C>0$ such that for any nonempty finite $F\subset S$, 
there is a finite measurable partition $\{B_0^F,B_1^F,\dotsc,B_C^F\}$ with $\mu(B_0^F)<\varepsilon \delta <\frac{\varepsilon^2}{2}$ and $\bar{d}_F(x,y)<\varepsilon \delta $ whenever $x,y \in B_j^F$ for some $j =1,\dotsc,C$. Let $U=(B_0^F \cup U_{0})^c $.
Then $\mu(U)>1-\varepsilon^2$.
For any such $F \subset S$ and $x\in X$, let $W_x=\{g \in F \colon gx \in U\}$ and 
\[
    W_F=\{x \in X \colon |W_x|>(1-\varepsilon)|F|\}.
\]
Then 
\begin{align*}
    \mu(U)=\int_X \frac{1}{|F|}\sum_{g \in F} 1_{U}(gx) d\mu
    &\leq \int_{W_F} \frac{1}{|F|}\sum_{g \in F} 1_{U}(gx) d\mu+\int_{X \setminus W_F} \frac{1}{|F|}\sum_{g \in F} 1_{U}(gx) d\mu \\ 
    & \leq \mu(W_F)+(1-\varepsilon)\mu(X \setminus W_F),
\end{align*}
which implies that $\mu(W_F)>1-\varepsilon$.
Now fix $x,y \in B^F_j\cap W_F \cap U_{0}^c $ for some $j=1,\dotsc,C$, and let $F(x,y)=\{g \in F :d(gx,gy)\geq\delta\}$. 
Since $\bar{d}_F(x,y)<\varepsilon \delta$, 
\[
    |F(x,y)|\delta \leq  \sum_{g \in F}d(gx,gy) <\varepsilon \delta |F|.
\]
Thus $|F(x,y)| < \varepsilon |F|$. 
Observe that for every $g \in F' := (F \setminus F(x, y)) \cap W_x \cap W_y$, we have $d(gx, gy) < \delta$ with $gx,gy \in U_0^c$.
Notice that $d(x, y) < \delta$ with $x, y \in U_{0}^c$ implies $H_\alpha(x, y)=0$.
Then
\begin{align*} 
    H_{F,\alpha}(x,y)
    &=\frac{1}{|F|}\biggl(\sum_{g \in F \setminus F' }H_{\alpha}(gx,gy)+\sum_{g \in F' }H_{\alpha}(gx,gy)\biggr) \\
    &\leq\frac{1}{|F|}|F \setminus F'| 
    \leq \frac{|F(x,y)|}{|F|} +\frac{|F\setminus W_x|}{|F|} +\frac{|F\setminus W_y|}{|F|} \\
    &< 3\varepsilon.
\end{align*}
Since $\mu(\bigcup_{j=1}^C B_j^F \cap W_F \cap U_{0}^c )>1-\varepsilon -\varepsilon^2>1-3\varepsilon$,  we have $\mathcal{C}_{\alpha}(F,3\varepsilon)\leq C$ for any nonempty finite subset $F$ of $S$. 
Hence $\alpha$ has bounded mean complexity along $S$. 
By the arbitrariness of $\alpha$, we have that $(X,\mathcal{B}_X,\mu,G)$ has bounded mean complexity along $S$.

($\Leftarrow$)
Without loss of generality, we assume
that $d(x,y)\leq 1$
for any $x,y \in X$.
Fix an arbitrary $\varepsilon\in (0,1)$.
Let $\alpha=\{A_1,\dotsc,A_r\}$ be a finite partition of $X$ such that each $A_j$ has diameter smaller than $\varepsilon$. 
Observe that $H_{\alpha}(x,y)=0$ or $1$. If $H_{\alpha}(x,y)=0$, then   $d(x,y)<\varepsilon$. 
Since $(X,\mathcal{B}_X,\mu,G)$ has bounded mean complexity along $S$, there exists $C\in \mathbb{N}$ such that for any nonempty finite $F\subset S$, there is a partition $\{B_0^F,B_1^F,\dotsc,B_C^F\}$ of $X$ with $\mu(B_0^F)<\varepsilon$ and 
$H_{F,\alpha}(x,y)<\varepsilon$ whenever $x,y \in B_j^F$ for some $j =1,\dotsc,C$.

Fix $x,y \in B^F_j$ for some $j=1,\dotsc,C$. 
Let $F(x,y)=\{g \in F \colon H_{\alpha}(gx,gy)= 1\}$. 
Since $H_{F,\alpha}(x,y)<\varepsilon$, $|F(x,y)|<\varepsilon|F|$.
Notice that $H_{\alpha}(gx,gy)=0 $ for any $g \in F \setminus F(x,y)$.
Then
\begin{align*}
    \bar{d}_F(x,y)&=\frac{1}{|F|}\left(\sum_{g \in F \setminus F(x,y)}d(gx,gy) + \sum_{g \in F(x,y)}d(gx,gy)\right) \\
    &\leq \varepsilon+\frac{|F(x,y)|}{|F|}<  2\varepsilon.
\end{align*}
Thus, $\mathcal{C}_d(F,2\varepsilon)\leq C$ for any nonempty finite subset $F$ of $S$.
\end{proof}

Let $(X,d,G)$ be a topological dynamical system and $G$ be a countably infinite discrete amenable group.
Let $\mu$ be a $G$-invariant measure on $X$ and $(F_n)_{n=1}^\infty$ be a F{\o}lner sequence in $G$.
We say that 
\begin{itemize}
\item $\mu$ is \emph{mean equicontinuous with respect to $(F_n)_{n=1}^\infty$} if for every $\tau>0$ there exists a Borel subset $X_\tau$ of $X$ with $\mu(X_\tau)>1-\tau$ such that 
 for any $\varepsilon>0$ there exists $\delta>0$ with the property that 
 for any $x,y\in X_\tau$ with $d(x,y)<\delta$,
 \[
     \limsup_{n\to\infty}\frac{1}{|F_n|}\sum_{g\in F_n}d(gx,gy)<\varepsilon;
 \]
\item $\mu$ is \emph{equicontinuous in the mean with respect to $(F_n)_{n=1}^\infty$}
 if for every $\tau>0$ there exists a Borel subset $X_\tau$ of $X$ with $\mu(X_\tau)>1-\tau$ such that 
 for any $\varepsilon>0$ there exists $\delta>0$ with the property that 
 for any $x,y\in X_\tau$ with $d(x,y)<\delta$,
 \[
    \frac{1}{|F_n|}\sum_{g\in F_n}d(gx,gy)<\varepsilon,\quad \forall n\in\mathbb{N};
 \]
\item $\mu$ has \emph{bounded $((F_n)_{n=1}^\infty,d)$-mean complexity}
 if for every $\varepsilon>0$ there exists $C=C(\varepsilon)>0$ such that 
 $\mathcal{C}_d(F_n,\varepsilon)\leq C$ for all $n\in\mathbb{N}$.
\end{itemize}

The following result is part of the main result in \cite[Theorem 1.3]{YZZ21}.
One can give a new proof of the following result using the idea in the proof of Theorems~\ref{thm:main2} and~\ref{thm:tds-mps-cp}.
We leave the details to the reader.

\begin{thm}
Let $(X,d,G)$ be a topological dynamical system and let $G$ be a countably infinite discrete amenable group.
Let $\mu$ be a $G$-invariant measure on $X$ and let $(F_n)_{n=1}^\infty$ be a tempered F{\o}lner sequence in $G$.
Then the following assertions are equivalent:
\begin{enumerate}
\item $(X,\mathcal{B}_X,\mu, G)$ has discrete spectrum;
\item $\mu$ is mean equicontinuous with respect to $(F_n)_{n=1}^\infty$;
\item $\mu$ is equicontinuous in the mean with respect to $(F_n)_{n=1}^\infty$;
\item $\mu$ has bounded $((F_n)_{n=1}^\infty,d)$-mean complexity.
\end{enumerate}
\end{thm}

\noindent \textbf{Acknowledgments}: The authors would like to thank the referees for their valuable comments and suggestions, which have helped to improve the presentation of this paper.
The authors were partially supported by National Key R\&D Program of China (2024YFA1013601, 2024YFA1013600), NSF of China (12222110),
NSF of Guangdong Province (2026A1515011212) and a grant from the Department of Education of Guangdong Province (2025KCXTD013).


\begin{thebibliography}{99}

\bibitem{AN23}
el Abdalaoui, el Houcein; Nerurkar, Mahesh. Weakly tame systems, their characterizations and applications. Monatsh. Math. 201 (2023), no. 3, 725--769. 

\bibitem{D11} 
Downarowicz, Tomasz. Entropy in dynamical systems. New Mathematical Monographs, 18. Cambridge University Press, Cambridge, 2011.

\bibitem{G17} 
García-Ramos, Felipe. 
Weak forms of topological and measure-theoretical equicontinuity: relationships with discrete spectrum and sequence entropy. 
Ergodic Theory Dynam. Systems 37 (2017), no. 4, 1211--1237.

\bibitem{GM19} 
García-Ramos, Felipe; Marcus, Brian. 
Mean sensitive, mean equicontinuous and almost periodic functions for dynamical systems, 
Discrete Contin. Dyn. Syst.  39 (2019), no.2, 729--746. 

\bibitem{G03} 
Glasner, Eli. Ergodic theory via joinings. Mathematical Surveys and Monographs, 101. American Mathematical Society, Providence, RI, 2003. 

\bibitem{F97}
Ferenczi, Sébastien. Measure-theoretic complexity of ergodic systems. Israel J. Math. 100 (1997), 189--207.

\bibitem{F51} 
Fomin, S. On dynamical systems with a purely point spectrum. (Russian) Doklady Akad. Nauk SSSR (N.S.) 77 (1951), 29--32. 

\bibitem{HN42} Halmos, Paul R.; von Neumann, John. Operator methods in classical mechanics. II. Ann. of Math. (2) 43 (1942), 332--350. 

\bibitem{HMXZ24} 
Hu, Zongrui; Ma, Xiao; Xu, Leiye; Zhou, Xiaomin.
Discrete spectrum of probability measures for locally compact group actions. arXiv:2412.17409 

\bibitem{HLTXY21}
Huang, Wen; Li, Jian; Thouvenot, Jean-Paul; Xu, Leiye; Ye, Xiangdong. Bounded complexity, mean equicontinuity and discrete spectrum. Ergodic Theory Dynam. Systems 41 (2021), no. 2, 494--533. 

\bibitem{HLY11} 
Huang, Wen; Lu, Ping; Ye, Xiangdong. Measure-theoretical sensitivity and equicontinuity. Israel J. Math. 183 (2011), 233--283. 

\bibitem{HSY05} 
Huang, Wen; Shao, Song; Ye, Xiangdong. Mixing via sequence entropy. Algebraic and topological dynamics, 101--122, Contemp. Math., 385, Amer. Math. Soc., Providence, RI, 2005. 

\bibitem{HWY19}
Huang, Wen; Wang, Zhiren; Ye, Xiangdong. Measure complexity and Möbius disjointness. Adv. Math. 347 (2019), 827--858. 

\bibitem{HY09}
Huang, Wen; Ye, Xiangdong. Combinatorial lemmas and applications to dynamics, Adv. Math 220 (2009), no. 6, 1689--1716.

\bibitem{K67} 
Kušnirenko, A. G. Metric invariants of entropy type. (Russian) Uspehi Mat. Nauk 22 (1967), no. 5(137), 57--65.

\bibitem{LSS24} 
Lenz, Daniel; Spindeler, Timo; Strungaru, Nicolae. Pure point spectrum for dynamical systems and mean, Besicovitch and Weyl almost periodicity. Ergodic Theory Dynam. Systems 44 (2024), no. 2, 524--568. 

\bibitem{LTY15} 
Li, Jian; Tu, Siming; Ye, Xiangdong. Mean equicontinuity and mean sensitivity. Ergodic Theory Dynam. Systems 35 (2015), no. 8, 2587--2612.

\bibitem{L01} 
Lindenstrauss, Elon. Pointwise theorems for amenable groups. Invent. Math. 146 (2001), no. 2, 259--295.

\bibitem{LWX25} 
Liu, Chunlin; Wang, Xiangtong; Xu, Leiye. Sequence entropy and IT-tuples for minimal group actions. Adv. Math. 467 (2025), Paper No. 110183, 34 pp. 

\bibitem{S82} 
Scarpellini, Bruno. Stability properties of flows with pure point spectrum. J. London Math. Soc. (2) 26 (1982), no. 3, 451--464. 

\bibitem{XX24} Xu, Fang; Xu, Leiye. Discrete spectrum for group actions. Dyn. Syst. 39 (2024), no. 1, 141--149. 

\bibitem{VZP13} 
Vershik, Anatoly M.; Zatitskiy, Pavel B.; Petrov, Fedor V. Geometry and dynamics of admissible metrics in measure spaces. Cent. Eur. J. Math. 11 (2013), no. 3, 379--400. 

\bibitem{VVZ23} 
Vershik, A. M.; Veprev, G. A.; Zatitskii, P. B. Dynamics of metrics in measure spaces and scaling entropy. (Russian) ; translated from Uspekhi Mat. Nauk 78 (2023), no. 3(471), 53--114 Russian Math. Surveys 78 (2023), no. 3, 443--499. 

\bibitem{Y19} 
Yu, Tao. Measure-theoretic mean equicontinuity and bounded complexity. J. Differential Equations 267 (2019), no. 11, 6152--6170. 

\bibitem{YZZ21} 
Yu, Tao; Zhang, Guohua; Zhang, Ruifeng. Discrete spectrum for amenable group actions. Discrete Contin. Dyn. Syst. 41 (2021), no. 12, 5871--5886.

\end{thebibliography}
\end{document}